\documentclass[10pt, twoside]{amsart}
\usepackage{amsmath,mathtools}
\usepackage{amsmath, amsthm, amssymb, amsfonts, enumerate}
\usepackage[colorlinks=true,linkcolor=blue,urlcolor=blue]{hyperref}
\usepackage{dsfont}
\usepackage{color}
\usepackage{geometry}
\usepackage{todonotes}
\usepackage{epstopdf}
\usepackage{bbm}
\usepackage{geometry}
\usepackage[utf8]{inputenc}
\usepackage{bm}

\usepackage{amsfonts}
\usepackage{amsfonts}
\usepackage{textcomp}
\usepackage{amssymb}
\usepackage{float}
\usepackage{tikz}
\usepackage{epsfig}
\usepackage{amsmath}
\usepackage[english]{babel}
\usepackage{a4}
\usepackage{enumerate}
\usepackage{tcolorbox}

\usepackage{soul}
\newcommand{\stkout}[1]{\ifmmode\text{\sout{\ensuremath{#1}}}\else\sout{#1}\fi}
\newtheorem{theorem}{Theorem}[section]

\newtheorem{remark}[theorem]{Remark}
\newtheorem{assumption}[theorem]{Assumption}
\newtheorem{lemma}[theorem]{Lemma}
\newtheorem{prop}[theorem]{Proposition}

\newtheorem{definition}[theorem]{Definition}

\newcommand{\cp}[2]{\langle#1,#2\rangle}
\newcommand{\ds}{\displaystyle}
\newcommand{\tr}{\mathop{\mathrm{Tr}}\nolimits}

\def \R{\mathbb{R}}

\def\erre{\mathbb{R}}

\definecolor{red}{rgb}{1.0,0.0,0.0}

\definecolor{blu}{rgb}{0.0,0.0,1.0}

\definecolor{gre}{rgb}{0.03,0.50,0.03}

\usepackage[T1]{fontenc}

\title[Stochastic optimal control problems with delays in the state and in the control]{Stochastic optimal control problems with delays in the state and in the control via viscosity solutions  and applications to optimal advertising and optimal 
investment problems}

\date{\today}
\author[De Feo]{Filippo De Feo}
\address{F.~De Feo: Department of Economics and Finance,  Luiss Guido Carli University, Rome 00197, Italy and Department of Mathematics, Politecnico di Milano, Piazza Leonardo da Vinci 32, 20133 Milano, Italy}
\email{\href{mailto:filippo.defeo@polimi.it}{fdefeo@luiss.it, filippo.defeo@polimi.it}}

\numberwithin{equation}{section}

\begin{document}

\begin{abstract}
In this manuscript we consider optimal control problems of stochastic differential equations with delays in the state and in the control. First, we prove an equivalent Markovian reformulation on Hilbert spaces of the state equation. Then, using the dynamic programming approach for infinite-dimensional systems, we  prove that the value function is the unique viscosity solution of the infinite-dimensional Hamilton-Jacobi-Bellman equation. We apply these results  to problems coming from economics:  stochastic optimal advertising problems  and stochastic optimal  investment problems with time-to-build.
\end{abstract}

\maketitle

\section{Introduction}
In this manuscript, we consider a class of stochastic optimal control problems with infinite horizon with delays in the state and in the control. Precisely, the state equation is a stochastic delay differential equation (SDDE) in $\mathbb R^n$ of the form
\begin{equation}\label{eq:SDDE_linear_intro}
\begin{cases}
dy(t) = \ds \left[a_0 y(t)+b_0 (u(t) )+ \int_{-d}^0 a_1(\xi)y(t+\xi)\,d\xi +
            \int_{-d}^0p_1(\xi)u(t+\xi)\,d\xi\right]dt 
\ds \\ \ \ \ \ \ \ \ \ \ \ \ + \sigma_0(u(t))\, dW(t), \quad t \geq 0, \\[10pt]
y(0)=\eta_0, \quad y(\xi)=\eta_1(\xi),\;
u(\xi)=\delta(\xi),\;\; \xi\in[-d,0),
\end{cases}
\end{equation}
where  $u(t)$ is the control process taking values in a bounded subset $U$ of $\mathbb R^p$,  $\eta_0 \in \mathbb{R}^n,$ $\eta_1 \in L^2([-d,0];\mathbb R^n)$  are the initial conditions of  $y(t)$,  $\delta \in L^2([-d,0); U)$ is the initial condition of $u(t)$, $W(t)$ is a Brownian motion with values in $\mathbb R^q$.
Denoting $\eta=(\eta_0,\eta_1)$, the goal is to minimize, over all $u(\cdot )\in \mathcal U$, a functional of the
form
\begin{equation*}
\mathcal J(\eta,\delta;u(\cdot)) =
\mathbb E\left[\int_0^{\infty} e^{-\rho t}
l(y(t),u(t)) dt \right].
\end{equation*}
The presence of delays is the crucial aspect of  (SDDE): these appear linearly via the integral terms
$$\int_{-d}^0 a_1(\xi)y(t+\xi)\,d\xi +
            \int_{-d}^0p_1(\xi)u(t+\xi)\,d\xi,$$
 where the first represents the delay in the state and the second the one in the control.
 For a complete picture of optimal control problems with delays we refer the reader to \cite{defeo_federico_swiech}, while here we will only recall some results.
 A similar problem with delays only in the state was treated by means of the dynamic programming method via viscosity solutions in \cite{defeo_federico_swiech}. In this paper we aim to extend some of  these results to the case of delays also in the control.

The main difficulty for delay problems is in the lack of Markovianity, which prevents a direct application of the dynamic programming method, e.g. see \cite{defeo_federico_swiech}. The approach we follow here, similarly to \cite{defeo_federico_swiech}, is to lift the state equation to an infinite-dimensional Banach or Hilbert space (depending on the regularity of the data), in order to regain Markovianity.\footnote{For the procedure of  rewriting deterministic delay differential equations, we refer the readers to \cite[Part II, Chapter 4]{delfour}. For the stochastic case, one may consult \cite{choj78, DZ96, gozzi_marinelli_2004, gozzi_marinelli_savin_2006, tankov_federico} for the Hilbert space case and
\cite{mohammed_1984, mohammed_1998, flandoli_zanco} for the Banach space case. A ``mixed'' approach is employed in  \cite{federico_2011}. }   This is done at a cost of of moving to infinite-dimension.

In  \cite{defeo_federico_swiech} (where delays are only in the state) the classical approach of rewriting the state equation in the Hilbert space $$X=:\mathbb R^n \times L^2([-d,0];\mathbb R^n)$$ was used. However when delays  in the control appear this procedure  has to be extended carefully. One possible way would be to use the extended delay semigroup (including the control in the state-space of the delay semigroup), e.g. see \cite[Part II, Chapter 4]{delfour} for deterministic problems. This approach brings the complication of having an unbounded control operator ("boundary control"). However, when the delays appear in a linear way in the state equation, an alternative approach can be used. Such approach  was  proposed by \cite{VK} for deterministic control problems  (see also \cite[Part II, Chapter 4]{delfour}) and extended in \cite{gozzi_marinelli_2004} to the stochastic setting.  In \cite{gozzi_marinelli_2004} a linear state equation with additive noise is considered and an equivalent abstract representation of the state equation in the Hilbert space $X$ is proved. In this paper we generalize this result of \cite{gozzi_marinelli_2004} to the following nonlinear state equation with multiplicative noise (note that such state equation is more general than \eqref{eq:SDDE_linear_intro}):
\begin{equation*}
\left\{\begin{array}{l}
dy(t) = \ds \left[b_0 ( y(t),u(t) )+ \int_{-d}^0 a_1(\xi)y(t+\xi)\,d\xi +
            \int_{-d}^0p_1(\xi)u(t+\xi)\,d\xi\right]dt \\[10pt]
\ds \qquad\qquad  + \sigma_0(y(t),u(t))\, dW(t), \quad t \geq 0 \\[10pt]
y(0)=\eta_0, \quad y(\xi)=\eta_1(\xi),\;
u(\xi)=\delta(\xi),\;\; \xi\in[-d,0),
\end{array}\right.
\end{equation*}
In this case the abstract state equation in $X$ is of the form 
\begin{equation*}
\begin{cases}
dY(t) = [A Y(t)+b(Y(t),u(t))]dt + \sigma(Y(t),u(t))\,dW(t), \quad \forall t \geq 0, \\[4pt]
Y(0) = x=M(\eta,\delta).
\end{cases}
\end{equation*}
for suitable  operators $A,M$ and  coefficients $b, \sigma$.  See Theorem \ref{th:equiv} for the equivalent abstract representation of the state equation in $X.$

Going back to our original problem, \eqref{eq:SDDE_linear_intro} can  be rewritten on $X$ as
\begin{equation}\label{eq:intro_state_eq_in_X}
\begin{cases}
dY(t) = [\mathcal A Y(t)+f(u(t))]dt + \sigma(u(t))\,dW(t), \quad \forall t \geq 0, \\[4pt]
Y(0) = x=M(\eta,\delta),
\end{cases}
\end{equation}
for suitable  $\mathcal A,f, \sigma$ (see \eqref{eq:abstract2}). The functional $\mathcal J(\eta,\delta;u(\cdot))$ is rewritten as
\begin{equation*}
J(x;u(\cdot)) :=
\mathbb E \left[ \int_0^{\infty} e^{-\rho t}L(Y(t),u(t))\,dt \right],
\end{equation*}
for a suitable cost function $L$  (see \eqref{ex1jbisbis}). Having lifted the problem in the space $X$ we are in a similar setting to the one of \cite{defeo_federico_swiech}, hence we wish to proceed in a similar way. Indeed we want to use the theory of viscosity solutions in Hilbert spaces (see \cite[Chapter 3]{fgs_book}) in order to treat the optimal control problem. However, $\mathcal A,f,\sigma$ in \eqref{eq:intro_state_eq_in_X} have a  different  structure than the ones in \cite{defeo_federico_swiech}. Indeed in \cite{defeo_federico_swiech} the unbounded operator  is the classical delay operator, while $\mathcal A$ here it is (up to a bounded perturbation) its adjoint and the coefficient $f$ here has a non-zero $L^2-$component.

At this point, in order to use the theory of viscosity solutions in Hilbert spaces  of \cite[Chapter 3]{fgs_book}, we rewrite the state equation  by introducing a maximal dissipative operator $\tilde{\mathcal A}$ in the state equation, see Proposition \ref{prop:Amaxdiss}. Next we introduce an operator $B$ satisfying the so called weak $B$-condition (which is crucial in the theory of viscosity solutions in Hilbert spaces), see Proposition \ref{prop:weak_B}. Hence, we prove that the data of the problem satisfy some regularity conditions with respect to the norm induced by the operator $B^{1 / 2}$, see Lemma \ref{lemma:regularity_coefficients_theory}. This enables us to characterize the value function of the problem as the unique viscosity solution of the following fully non linear second order infinite-dimensional HJB equation
\begin{equation*}
%\left\{\begin{array}{l}
\rho v(x) - \cp{\tilde{\mathcal A} x}{Dv(x)}_X + H(x,Dv(x),D^2v(x))=0,
\quad \forall x \in X,
%0 \leq t \leq T\\[8pt]
%v(T) = \varphi,
%end{array}\right.
\end{equation*}
where $H$ is the Hamiltonian. See Theorem \ref{th:existence_uniqueness_viscosity_infinite} for such result. To the best of our knowledge, this is the first existence and uniqueness result for fully nonlinear HJB equations in Hilbert spaces related to  stochastic optimal control problems with delays in the state and in the control. This extends the corresponding result of \cite{defeo_federico_swiech} to the case of delays (also) in the control. 
Moreover  in the present paper the delay kernels $a_1,p_1$ are such that their rows  $a_1^j,p_1^j  \in L^2([-d,0]; \mathbb R^n)$ for every $j \leq n$. Instead  in \cite{defeo_federico_swiech} a higher regularity of $a_1,a_2$ (where $a_2$ is the delay kernel associated to the diffusion) is required, i.e. $a_1^j ,a_2^j \in W^{1,2}([-d,0]; \mathbb R^n)$ and $a_1^j(-d)=a_2^j(-d)=0$ for every $j \leq n$  (of course $p_1=0$, i.e. only delays in the state are present). However we remark that the structure of the state equation in \cite{defeo_federico_swiech}, with delays only in the state, is more general.

For a complete picture of the literature related to optimal control problems with delays only in the state we refer to \cite{defeo_federico_swiech}, here we only list some results: for  stochastic optimal control problems with delays only in the state e.g. see  \cite{biffis_gozzi_prosdocini_2020}, \cite{biagini_gozzi_zanella_2022}, \cite{deFeoSwiech}, \cite{djehiche_gozzi_zanco_zanella_2022}, \cite{dizacinto_federico_gozzi_2011}, \cite{federico_2011},  \cite{tankov_federico}, \cite{fuhrman_tessitore_masiero_2010}, \cite{masiero_2022}, \cite{pang_2019}.
For deterministic optimal control problems with delays only in the state e.g. see \cite{carlier_tahraoui_2010}, \cite{goldys_1}, \cite{goldys_2}.

 Stochastic differential equations with delays also in the control are more difficult since in this case the so called structure condition, that is the requirement that the range of the control operator is contained in the range of the noise, does not hold, e.g. see \cite[Subsection 2.6.8]{fgs_book}. This fact, together with  the lack of smoothing of the transition semigroup associated to the linear part of the equation (this is a common feature also in problems with delays only in the state), prevent the use of standard techniques, based on mild solutions and on backward stochastic differential equations. However stochastic differential equations with delays only in the control, linear structure of the state equation and additive noise (the HJB equation is semilinear) were completely solved in \cite{gozzi_masiero_2017}, \cite{gozzi_masiero_2017_b}, \cite{gozzi_masiero_2022} by means of a partial smoothing property for the Ornstein–Uhlenbeck transition semigroup which permitted to apply a variant of the approach via mild solutions in the space of continuous functions. See also \cite{gozzi_masiero_2022b}, \cite{gozzi_masiero_2023}, where this approach is extended to stochastic optimal control problems with unbounded control operators and applications to problems with delays only in the control (with delay kernel being a Radon measure) are given. 
Finally we refer to \cite{tacconi} for a deterministic optimal control problem with delays only  in the control and linear structure of the state equation solved by means of viscosity solutions in the space $\mathbb R \times W^{1,2}([-d,0])$  (see also Remark \ref{rem:x_0_leq_x_-1}). 

At the end of the manuscript, we provide applications of our results to problems coming from economics. We consider a stochastic optimal advertising problem with delays in the state and in the control and controlled diffusion, generalizing the one from \cite{gozzi_marinelli_2004}, \cite{gozzi_marinelli_savin_2006}  (see also \cite{nerlove_arrow} for the original deterministic model). We characterize the value function as the unique viscosity solution of the fully non-linear HJB equation. We recall that, in the stochastic setting with additive noise, in \cite{gozzi_marinelli_2004}, \cite{gozzi_marinelli_savin_2006}  a verification theorem was proved in the context of classical solutions  and optimal feedback strategies were derived.   Moreover an explicit (classical) solution of the HJB equation was derived in a specific case. We also recall that in \cite{defeo_federico_swiech} the case with no delays in the control (i.e. $p_1=0$) was treated via viscosity solutions.  \\
Finally, we consider a stochastic optimal investment problem with with time-to-build, inspired by  \cite[p. 36]{fabbri_federico} (see also e.g. \cite{bambi_fabbri_gozzi (2012)}, \cite{bambi_digirolami_federico_gozzi (2017)} for similar models in the deterministic setting). We characterize the value function as the unique viscosity solution of the fully non-linear HJB equation.

The paper is organized as follows. In Section \ref{sec:formul} we introduce the problem and state the main assumptions. In Section \ref{sec:reformulation_infinite_dim} we prove an equivalent infinite-dimensional formulation for a more general state equation and we rewrite the problem in an infinite dimensional setting. In Section \ref{sec:B_continuity} we prove some preliminary estimates for solutions of the state equation and the value function. In Section \ref{sec:th_uiqueness} we introduce the notion of viscosity solution of the HJB equation and state a theorem about the existence and uniqueness of viscosity solutions, and characterize the value function as the unique viscosity solution. In  Section \ref{sec:advertising} we provide  applications to problems coming from economics: stochastic optimal advertising models and stochastic investment models with time-to-build.

%At the end of the chapter, we provide applications to optimal advertising problems with delays in the state and in the control and investment models with time-to-build.

\section{Setup and assumptions}\label{sec:formul}
We denote by  $M^{m \times n}$ the space of real valued $m \times n$-matrices and we denote by $|\cdot|$ the Euclidean norm in $\mathbb{R}^{n}$ as well as the norm of elements of $M^{m \times n}$ seen as linear operators from $\mathbb{R}^{m}$ to  $\mathbb{R}^{n}$. We will write $x \cdot y$ for the inner product in $\mathbb R^n$.

Let $d>0$.
We consider 
the standard Lebesgue space $L^2:=L^2([-d,0];\mathbb{R}^n)$ of square integrable functions from  $[-d,0]$ to $\erre^{n}$, denoting by  
$\langle \cdot,\cdot\rangle_{L^{2}}$ the inner product in $L^2$ and by $|\cdot|_{L^{2}}$ the norm. We also consider the standard Sobolev space $W^{1,2}:=W^{1,2}([-d,0]; \mathbb{R}^n)$ of functions in $L^{2}$ admitting weak derivative in $L^{2}$, endowed with the inner product $\langle f,g\rangle_{W^{1,2}}:= \langle f,g\rangle_{L^{2}}+\langle f',g'\rangle_{L^{2}}$ and norm $|f|_{W^{1,2}}:=(|f|^2_{L^{2}}+|f'|^2_{L^{2}})^{\frac{1}{2}}$, which render it a Hilbert space. It is well known that the space $W^{1,2}$ can be identified with the space of absolutely continuous functions from $[-d,0]$ to $\R^{n}$.

Let $\tau=\left(\Omega, \mathcal{F},\left(\mathcal{F}_t\right)_{t \geq 0}, \mathbb{P}, W\right)$ be a reference probability space, that is $(\Omega, \mathcal{F}, \mathbb{P})$ is a complete probability space, $W=(W(t))_{t \geq 0}$ is a standard $\mathbb{R}^q$-valued Wiener process, $W(0)=0$, and $\left(\mathcal{F}_t\right)_{t \geq 0}$ is the augmented filtration generated by $W$. We consider the following controlled stochastic delay differential equation (SDDE) 
\begin{equation}
\label{eq:SDDE_2}
\begin{cases}
dy(t) = \ds \left[a_0 y(t)+b_0 (u(t) )+ \int_{-d}^0 a_1(\xi)y(t+\xi)\,d\xi +
            \int_{-d}^0p_1(\xi)u(t+\xi)\,d\xi\right]dt 
\ds \\ \ \ \ \ \ \ \ \ \ \ \ + \sigma_0(u(t))\, dW(t), \quad t \geq 0, \\[10pt]
y(0)=\eta_0, \quad y(\xi)=\eta_1(\xi),\;
u(\xi)=\delta(\xi),\;\; \xi\in[-d,0).
\end{cases}
\end{equation} 
where 
\begin{enumerate}[(i)]
\item given a bounded measurable  set $U\subset \R^{p}$,
$u(\cdot)$ is the control process lying in the set 
\[
\mathcal{U}_\tau=\{u(\cdot): \Omega\times [0,+\infty )\to U: \	 u(\cdot) \ \mbox{is} \ (\mathcal{F}_t)\mbox{-progressively measurable and integrable a.s.} \};
\]
\item $\eta_0 \in \mathbb{R}^n$ and $\eta_1 \in L^2$  are the initial conditions of the state $y$;
\item $\delta \in L^2([-d,0); U)$ is the initial condition of the control $u(\cdot);$
\item $b_0 \colon  U \to \mathbb{R}^n$, $\sigma_0 \colon  U \to M^{n\times q}$;
\item $a_{1}, p_1 :[-d,0]\to M^{n \times n}$ 
and if $a_{1}^{j},p_1^j$ are the $j$-th row of $a_1(\cdot)$, $p_1(\cdot)$ respectively  for $j=1,...,n$, then  $a_1^{j}, p_1^{j}\in L^2([-d,0]; \mathbb R^n)$.
\end{enumerate}
\begin{remark}
Similarly to \cite{defeo_federico_swiech} we cannot treat the case of pointwise delay (e.g. $a_1,p_1=\delta_{-d}$ the Dirac's Delta).
In \cite{defeo_federico_swiech} a higher regularity of $a_1,a_2$ (where $a_2$ is the delay kernel associated to the diffusion) is required, i.e. $a_1^j ,a_2^j \in W^{1,2}([-d,0]; \mathbb R^n)$ and $a_1^j(-d)=a_2^j(-d)=0$ for every $j \leq n$  (of course $p_1=0$, i.e. in \cite{defeo_federico_swiech} only delays in the state are present). See also \cite{goldys_1}, \cite{goldys_2} for similar restrictions in deterministic problems. Here, instead, we require less regularity, i.e. $a_1^{j}, p_1^{j}\in L^2([-d,0]; \mathbb R^n)$.
\end{remark}
We will assume the following conditions.
\begin{assumption}\label{hp:state} 
$b_0 \colon  U \to \mathbb{R}^n$, $\sigma_0 \colon  U \to M^{n\times q}$
are continuous and bounded.
\end{assumption}
Under Assumption \ref{hp:state}, by \cite[Theorem IX.2.1]{revuz}, for each initial data $\eta:=(\eta_0,\eta_1)\in \mathbb{R}^{n}\times L^2([-d,0];\mathbb R^n)$, $\delta\in  L^2([-d,0);U)$, and each control $u(\cdot )\in\ \mathcal{U}$,  there exists a unique (up to Indistinguishability) strong solution to \eqref{eq:SDDE} and this solution admits a version with continuous paths that we denote by  $y^{\eta,\delta;u}$.

We consider the following infinite horizon optimal control problem. Given $\eta=(\eta_0,\eta_1)\in \mathbb{R}^{n}\times L^2$, $\delta \in L^2([-d,0); U)$ , we define a cost functional of the form
\begin{equation}
\label{eq:obj-origbis}
\mathcal J(\eta,\delta;u(\cdot)) =
\mathbb E\left[\int_0^{\infty} e^{-\rho t}
l(y^{\eta,\delta,u}(t),u(t)) dt \right]
\end{equation}
where $\rho>0$ is the discount factor, 
$l \colon \mathbb{R}^n \times U \to \mathbb{R} $ is the running cost. As in \cite[p. 98, Equation (2.8)]{fgs_book}, we define
\[
\mathcal{U}=\bigcup_{\tau}\mathcal{U}_\tau,
\] 
where the union is taken over all reference probability spaces $\tau$.
 The goal is to minimize $\mathcal{J}(\eta,\delta;u(\cdot))$ over all $u(\cdot)\in \mathcal{U}$. 
This is a standard setup of a stochastic optimal control problem (see \cite{yong_zhou, fgs_book}) used to apply the dynamic programming approach. We remark (see e.g. \cite{fgs_book}, Section 2.3.2) that 
\[
\inf_{u(\cdot)\in \mathcal{U}}\mathcal J(\eta,\delta;u(\cdot))=\inf_{u(\cdot)\in \mathcal{U}_\tau} \mathcal J(\eta,\delta;u(\cdot))
\]
for every reference probability space $\tau$ so the optimal control problem is in fact independent of the choice of a reference probability space.
\begin{assumption}\label{hp:cost} 
\begin{itemize}
\item[(i)] $l \colon \mathbb{R}^n \times U \to \mathbb{R} $ is continuous.
\item[(ii)]
There exist constants $K,m>0$, such that
\begin{equation}\label{eq:growth_l}
|l(z,u)| \leq K(1+|z|^m) \ \ \ \forall y\in  \mathbb{R}^n, \ \forall u \in U.
\end{equation}
\item[(iii)] There exists a local modulus of continuity for $l(\cdot,u)$, uniform in $u\in U$, i.e. for each $R>0$, there exists a nondecreasing concave function $\omega_{R}:\mathbb{R}^{+}\to \mathbb{R}^{+}$ such that   $\lim_{r\to 0^{+}} \omega_{R}(r)=0$ and 
\begin{equation}\label{eq:regularity_l}
|l(z,u)-l(z',u)| \leq \omega_R( |z-z'|)
\end{equation}
for every $z,z' \in \mathbb{R}^n$ such that $|z|,|z'| \leq R$ and every $u \in U$.
\end{itemize}
\end{assumption}
We will show, suitably reformulating the state equation in an infinite dimensional framework, that the cost functional is well defined and finite for a sufficiently large  discount factor $\rho>0$.

Throughout the paper we will write $C>0,\omega, \omega_R$ to indicate, respectively, a constant, a modulus continuity, and a local modulus of continuity, which may change from place to place if the precise dependence on other data is not relevant. The equality involving random variables will be intended $\mathbb P-$a.s..
\section{Infinite dimensional Markovian representation}\label{sec:reformulation_infinite_dim}
The optimal control problem at hand is not Markovian due to the delay. In order to regain Markovianity and approach the problem by Dynamic Programming we reformulate the state equation in an infinite-dimensional space generalizing a well-known procedure, see \cite[Part II, Chapter 4]{delfour}, \cite{VK} for deterministic delay equations and \cite{gozzi_marinelli_2004} for the stochastic case with linear state equation and  additive noise.

In this section, in place of \eqref{eq:SDDE_2}, we will consider the following more general state equation:
\begin{equation}
\label{eq:SDDE}
\left\{\begin{array}{l}
dy(t) = \ds \left[b_0 ( y(t),u(t) )+ \int_{-d}^0 a_1(\xi)y(t+\xi)\,d\xi +
            \int_{-d}^0p_1(\xi)u(t+\xi)\,d\xi\right]dt \\[10pt]
\ds \qquad\qquad  + \sigma_0(y(t),u(t))\, dW(t), \quad t \geq 0 \\[10pt]
y(0)=\eta_0, \quad y(\xi)=\eta_1(\xi),\;
u(\xi)=\delta(\xi),\;\; \xi\in[-d,0),
\end{array}\right.
\end{equation}
where, in this case, $b_0 \colon \mathbb{R}^n  \times U \to \mathbb{R}^n$, $\sigma_0 \colon \mathbb{R}^n \times U \to M^{n\times q}$, while all the other terms satisfy the same conditions as in \eqref{eq:SDDE_2}.
In this setting we consider the following assumptions.
\begin{assumption}\label{hp:state2} 
$b_0 \colon \mathbb{R}^n  \times U \to \mathbb{R}^n$, $\sigma_0 \colon \mathbb{R}^n \times U \to M^{n\times q}$
are continuous and  there exist constants $L,C>0$ such that
\begin{align*}
& |b_0(y,u)-b(y',u)|\leq L |y-y'|, \\
&|b_0(y,u)|\leq C(1+|y|), \\
&|\sigma_0(y,u)-\sigma_0(y',u)| \leq L |y-y'|, \\
&|\sigma_0(y,u)| \leq C(1+|y|),
\end{align*}
for every $y,y' \in \mathbb{R}^n$, $u \in U$.
\end{assumption}
Under Assumption \ref{hp:state2}, by \cite[Theorem IX.2.1]{revuz}, for each initial data $\eta:=(\eta_0,\eta_1)\in \mathbb{R}^{n}\times L^2([-d,0];\mathbb R^n)$, $\delta\in  L^2([-d,0);U)$, and each control $u(\cdot )\in\ \mathcal{U}$,  there exists a unique (up to indistinguishability) strong solution to \eqref{eq:SDDE} and this solution admits a version with continuous paths that we denote by  $y^{\eta,\delta;u}$.

We define 
$
X := \mathbb{R}^n \times L^2
$.
An element $x\in X$ is 
a couple $x= (x_0,x_1)$, where $x_0 \in \mathbb{R}^n$, $x_{1}\in L^{2}$; sometimes, we will write $x=\begin{bmatrix}
x_0\\ x_1
\end{bmatrix}.$ 
The space $X$ is a Hilbert space when endowed
with the inner product 
\begin{align*}
\langle x,z\rangle_{X}& := x_0\cdot  z_0 + \langle x_1,z_1 \rangle_{L^2}= x_0  z_0 + \int_{-d}^0  x_1(\xi)\cdot z_1(\xi) \,d\xi, \ \ \ x=(x_{0},x_{1})\in X, \ z=(z_{0},z_{1})\in X.
\end{align*}
The  induced norm, denoted by $|\cdot|_X$, is then
$$
|x|_{X} = \left( |x_0|^2 + \int_{-d}^0 |x_1(\xi)|_{L^{2}}^2\,d\xi\right)^{1/2}, \ \ \ x=(x_0,x_1) \in X.
$$
For $R>0$, we  set  the following notation for the open balls of radius $R$ in $X$, $\mathbb R^n,$ and $L^2$, respectively:
$$B_R:=\{x \in X: |x|_{X} < R\}, \ \ \ B_R^0:=\{x_0 \in \mathbb R^n: |x_0| < R\}, \ \ \ B_R^1:=\{x_1 \in L^2[-d,0]: |x_1|_{L^{2}} < R\},$$

We denote by 
$\mathcal{L}(X)$ the space of bounded linear operators from $X$ to $X$, endowed with the operator norm 
$$|L|_{\mathcal{L}(X)}=\sup_{|x|_{X}=1} |Lx|_{X}.$$
An operator $L \in \mathcal{L}(X)$ can be seen as 
$$
Lx=\begin{bmatrix}
L_{00} & L_{01}\\
L_{10} & L_{11}
\end{bmatrix}\begin{bmatrix}
x_0\\
x_1
\end{bmatrix}, \quad x=(x_0,x_1) \in X,
$$
where $L_{00} \colon \mathbb R^n \to \mathbb R^n$, $L_{01} \colon   L^2 \to \mathbb R^n$, $L_{10} \colon \mathbb R^n \to L^2$, $L_{11} \colon L^2 \to L^2$ are bounded linear operators.
Moreover, given two separable Hilbert spaces $(H, \langle \cdot , \cdot \rangle_H), (K, \langle \cdot , \cdot \rangle_K)$, we denote by $\mathcal{L}_1(H,K)$ the space of trace-class operators endowed with the norm 
$$|L|_{\mathcal{L}_1(H,K)}=\inf \left \{\sum_{i \in \mathbb N} |a_i|_H |b_i|_{K}: Lx=\sum_{i \in \mathbb N} b_i  \langle a_i,x \rangle, a_i \in H, b_i \in  K , \forall i \in \mathbb N \right \}.$$
We also denote by  $\mathcal{L}_2(H,K)$ the space of Hilbert-Schmidt operators from $H$ to $K$ endowed with the norm
$$|L|_{\mathcal{L}_2(H,K)}=(\operatorname{Tr}(L^*L))^{1/2}.$$ When $H=K$ we simply write $\mathcal{L}_1(H)$, $\mathcal{L}_2(H)$.
We denote by $S(H)$ the space of self-adjoint operators in $\mathcal{L}(H)$. If $Y,Z\in S(H)$, we write $Y\leq Z$ if $\langle Yx,x \rangle\leq \langle Zx,x \rangle$ for every $x \in H$.

Let us define the unbounded linear operator $A:D(A)\subset X\to X$ as
follows:
\begin{align*}
& A x = \begin{bmatrix}
x_1(0) \\
-x_1'
\end{bmatrix} ,    \ \ \ \ D(A) = \left\{ x \in X:
       x_1 \in W^{1,2}([-d,0]; \mathbb{R}^n), \ x_1(-d)=0\right\}.
\end{align*}
The adjoint of  $A$ is the operator $A^*:D(A^*)\subset X\to X$ (e.g., see \cite[Proposition 3.4]{goldys_1}) 
\begin{align*}
& A^*x= \begin{bmatrix} 0 \\
   x_1'\end{bmatrix}, \quad  D(A^*) = \left\{ x \in X: x_1 \in W^{1,2}([-d,0],\mathbb{R}^n), \ x_1(0)=x_0\right\}.
\end{align*}
Note that $A^*$ is the  generator of the delay semigroup, see, e.g.,  \cite[Part II, Chapter 4]{delfour}. For problems with delays (also) in the control appearing in a linear way in the state equation, its adjoint, i.e. $A$, is used to reformulate the problem in the space $X$ (see, e.g.,  \cite[Part II, Chapter 4]{delfour}). Indeed,  $A$ is the generator of a strongly continuous semigroup $e^{tA} $ on $X$, whose explicit expression (see, e.g.,  \cite[Eq. (73)]{tacconi}) is
\begin{equation}\label{eq:expression_semigroup}
e^{At}x=\begin{bmatrix} x_0+\int_{(-t) \vee  (-d)}^0 x_1(\xi)d \xi,\\\Phi(t)x_1 \end{bmatrix}, \quad x =(x_0,x_1) \in X,
\end{equation}
where $\Phi(t)$ is the semigroup of truncated right shift in $L^{2}$:
$$
[\Phi(t) f](\xi)=1_{[-d,0]}(\xi-t) f(\xi-t) \quad \forall f \in L^2.
$$
%Moreover there exist constant (SF: DOVE SI USA? SCRIVERE ESPLICITAMENTE M) $M>0$ such that
%\begin{equation}\label{est_sem}
%|e^{tA}|_{\mathcal{L}(X)}\leq M, \ \ \forall t\geq 0.
%\end{equation}
Now define $b \colon X \times U \to X$ by
$$b(x,u)=\begin{bmatrix}
b_0(x_0,u)\\
a_1x_0+p_1 u
\end{bmatrix} \quad \forall x=(x_0,x_1) \in X, u \in U$$ 
and  $\sigma \colon X \times U \to \mathcal{L}(\mathbb{R}^q,X)$ by 
$$\sigma(x,u)w=\begin{bmatrix} 
\sigma_0(x_0,u)\\
0
\end{bmatrix}
  \quad \forall x=(x_0,x_1) \in X, u \in U, w \in \mathbb{R}^q. $$ 
 
Consider the infinite dimensional SDE
\begin{equation}
\label{eq:abstract}
\begin{cases}
dY(t) = [A Y(t)+b(Y(t),u(t))]dt + \sigma(Y(t),u(t))\,dW(t), \\[4pt]
Y(0) = x \in X.
\end{cases}
\end{equation}
By \cite[Theorem 1.127]{fgs_book},  for each $u(\cdot) \in\mathcal{U}$, \eqref{eq:abstract} admits a unique mild solution; that is, there exists a unique progressively measurable   $X-$valued process $Y=(Y_{0},Y_{1})$ such that 
\begin{equation}\label{eq:mild1}
Y(t) =e^{{tA}}x+\int_{0}^{t}  e^{(t-s)A} b(Y(s),u(s)) d s+\int_{0}^{t} e^{(t-s)A}\sigma(Y(s),u(s))d W(s).
\end{equation}
Define  the linear  operator $M \colon X \times L^2([-d,0],U) \to X$ by
\begin{align*}
M (\alpha,\beta) := (\alpha_0,m(\alpha_1,\beta) ), \ \ \alpha=(\alpha_0,\alpha_1)\in X, \ \beta\in L^2([-d,0],U), 
\end{align*}
where
$$
m(\alpha_1,\beta) (\xi) := \int_{-d}^\xi a_1(\zeta) \alpha_1(\zeta-\xi)\,d\zeta
+ \int_{-d}^\xi p_1(\zeta) \beta(\zeta-\xi)\,d\zeta, \ \ \ \ \xi \in [-d,0].$$
\begin{theorem}\label{th:equiv} Let Assumption \ref{hp:state2} hold.
 We have the following claims.

\begin{enumerate}[(i)]
\item   Let $Y^{x,u}$ be the unique mild solution to \eqref{eq:abstract}
with  initial datum $x \in X$ and control
$u(\cdot)\in\mathcal{U}$. For every  $t \geq d$ 
$$
Y^{x,u}(t) =(Y^{x,u}_0(t),Y^{x,u}_1(t))= M \left  ( \left (Y^{x,u}_0(t),Y^{x,u}_0(t+\cdot) \right ),u(t+\cdot) \right ).$$
\item 
Let $y^{\eta,\delta;u}$ be the solution 
to SDDE \eqref{eq:SDDE} with initial data $\eta,\delta$ and under the control $u(\cdot)\in\mathcal{U}$, and let
$
x = M(\eta_{0},\eta_{1},\delta).
$
Then, for every $t\geq 0$,
$$
Y^{x,u}(t) =(Y^{x,u}_0(t),Y^{x,u}_1(t)) = M \left ( \left (y^{\eta,\delta,u}(t),y^{\eta,\delta, u}(t+\cdot) \right ),u(t+\cdot) \right ).
$$
In particular, for every $t\geq 0$,
$$y^{\eta,\delta,u}(t)=Y^{x,u}_0(t).$$
\end{enumerate}
\end{theorem}
\begin{proof} 
\begin{enumerate}[(i)]
\item 
Using \eqref{eq:expression_semigroup}, we can rewrite the two components of \eqref{eq:mild1}  as
\begin{small}
\begin{align*}
\begin{bmatrix}
Y_0^{x,u}(t)\\
\\
Y_1^{x,u}(t)
\end{bmatrix}
=\begin{bmatrix}
x_0 +\int_{(-t) \vee  (-d)}^0 x_1(\xi)d \xi + \int_0^t \left[ b_0(Y_0^{x,u}(s),u(s))+ \int_{(-(t-s)) \vee  (-d)}^0 a_1(\xi)Y_0^{x,u}(s)+ p_1(\xi)u(s) d \xi \right] ds\\
 + \int_0^t  \sigma_0(Y_0^{x,u}(s),u(s)) dW(s)\\\\
\Phi(t) x_{1}+\int_{0}^{t}\Phi(t-s) a_{1}Y_{0}^{x,u}(s)d s+\int_{0}^{t}\Phi(t-s) p_{1} u(s)d s
\end{bmatrix}.
\end{align*}
\end{small}
Then, 

\begin{align}\label{eq:proof_gozzi_marinelli_Y1_rewritten}
Y_{1}^{x,u}(t)(\xi) =&\ 1_{[-d,0]}(\xi-t) x_1(\xi-t)+\int_{0}^{t}1_{[-d,0]}(\xi-t+s)  a_{1}(\xi-t+s) Y_{0}^{x,u}(s) d s\nonumber\\
&+\int_{0}^{t} 1_{[-d,0]}(\xi-t+s)  p_{1}(\xi-t+s) u(s) ds \nonumber \\
=&\ 1_{[-d,0]}(\xi-t) x_1(\xi-t)+\int_{(\xi-t)\vee -d}^{\xi}  a_{1}(\alpha) Y_{0}^{x,u}(t+\alpha-\xi) d \alpha\\
&+\int_{(\xi-t)\vee -d}^{\xi} p_{1}(\alpha) u(t+\alpha-\xi) d\alpha.\nonumber
\end{align}

For $t \geq  d$, we have $\xi-t \leq  -d$, so that
$$
Y_{1}^{x,u}(t)(\xi)=\int_{-d}^{\xi} a_{1}(\alpha) Y_{0}^{x,u}(t+\alpha-\xi) d \alpha+\int_{-d}^{\xi} p_{1}(\alpha) u(t+\alpha-\xi) d \alpha =m\left (Y^{x,u}_{0}(t+\xi), u(t+\xi) \right)
$$
from which we get the first claim.
%$$
%Y^{x,u}(t)=\left(Y^{x,u}_{0}(t), m\left (Y^{x,u}_{0}(t+\cdot), u(t+\cdot) \right) \right)=M\left( \left (Y^{x,u}_{0}(t), Y^{x,u}_{0}(t+\cdot) \right), u(t+\cdot)\right) \ \ \forall t\geq d.
%$$

\medskip

%Note that the term $M(Y_0(t),Y_0(t+\cdot),u(t+\cdot))$ is well posed for every $t \geq 0$ since $Y_0$ is continuous and $u \in \mathcal{U}$.

\item  Let $x=(x_0,x_1)=M\left(\eta_{0}, \eta_1, \delta\right)$.
For $\xi-t \in[-r, 0]$, $\xi \in[-r, 0]$ we have:
$$
x_1(\xi-t)=Y^{x,u}_{1}(0)(\xi-t) =\int_{-d}^{\xi-t} a_{1}(\alpha) \eta(t+\alpha-\xi) d \alpha+\int_{-d}^{\xi-t} p_{1}(\alpha) \delta(t+\alpha-\xi) d \alpha,
$$
so that inserting it into \eqref{eq:proof_gozzi_marinelli_Y1_rewritten} we have:
\begin{align}\label{eq:proof_gozzi_marinelli_Y1_rewritten_2}
Y_{1}^{x,u}(t)(\xi)=\int_{-d}^{\xi} a_{1}(\alpha) \tilde{Y}^{x,u}_{0}(t+\alpha-\xi) d \alpha+\int_{-d}^{\xi} p_{1}(\alpha) u(t+\alpha-\xi) d \alpha,
\end{align}
where we have defined $\tilde Y^{x,u}_0$ to be the extension of $Y^{x,u}_0$ to $[-d,0)$ by
$$
\tilde{Y}^{x,u}_{0}(s)= \begin{cases}\eta_1(s), & s \in[-d, 0), \\ Y^{x,u}_{0}(s), & s \geq 0.\end{cases}
$$
From \eqref{eq:proof_gozzi_marinelli_Y1_rewritten_2} we have:
\begin{align}\label{eq:M_proof}
Y^{x,u}(t)=(Y^{x,u}_0(t),Y^{x,u}_1(t))=M(\tilde{Y}_{0}^{x,u}(t),\tilde{Y}_{0}^{x,u}(t+\cdot),u(t+\cdot))=(\tilde{Y}_{0}^{x,u}(t),m(\tilde{Y}_{0}^{x,u}(t+\cdot),u(t+\cdot)).
\end{align}
To conclude the proof, by uniqueness of strong solutions to \eqref{eq:SDDE}, we need to prove that $\tilde Y_{0}^{x,u}$ satisfies \eqref{eq:SDDE}.
%First we notice that, $y^{\eta,\delta,u(\cdot)}(t)$ satisfies for every $t \geq 0$
%\begin{align}\label{eq:proof_ssolution_SDDE}
%y^{\eta,\delta,u(\cdot)}(t) & =\eta_0 + \int_0^t \left[\int_{-d}^{0} a_{1}(\xi) y^{\eta,\delta,u(\cdot)}(s+\xi) d \xi+\int_{-d}^{0} p_{1}(\xi) u(s+\xi) d \xi\right]  ds \nonumber \\
%&\quad  +\int_{0}^{t}b_0(y^{\eta,\delta,u(\cdot)}(s), u(s)) d s+\int_{0}^{t}  \sigma_0(y^{\eta,\delta,u(\cdot)}(t),u(t))  d W(s),
%\end{align}
%where $y^{\eta,\delta,u(\cdot)}(\xi)=\eta(\xi)$ for every $\xi \in [-d,0)$.
On the other hand, by \cite[Theorem 3.2]{gawarecki}, $Y^{x,u}(t)$ is also a weak solution of \eqref{eq:abstract}, i.e. it satisfies
\begin{align*}
\langle Y^{x,u}(t),h\rangle_X
=\ &\langle x,h\rangle_X +\int_0^t \langle Y^{x,u}(t),A^* h\rangle_X ds +\int_0^t \langle b(Y^{x,u}(t),u(t)), h\rangle_X ds\\& +\int_0^t \langle \sigma(Y^{x,u}(t),u(t)), h \rangle_X dW(s),  \ \ \ \ \forall t \geq 0, \ \forall h \in D(A^*).
\end{align*}
For $k\in\mathbb{N}\setminus\{0\}$, let $h^k=(1,h_1^k)$ be defined by 
$$h_1^k(\xi)= \int_{-d}^\xi k \mathbf{1}_{[-1/k,0]}(r)dr, \ \ \ \ \ \xi\in[-d,0].$$  
Then, 
$$
\begin{cases}
h^k\in D(A^*)\ \ \ \forall k\in\mathbb{N}\setminus\{0\},\\
h_1^k(\xi)\to 0 \ \ \mbox{ as } \ k \to \infty \ \ \mbox{ for a.e.} \ \xi \in [-d,0],\\
\lim_{k \to \infty}\int_{-d}^0 \frac{dh_1^k}{d\xi}(\xi) z(\xi) d \xi = z(0), \ \ \ \forall z \in {C}([-d,0];\R^{n}).
\end{cases}
$$
Therefore, for every $t \geq 0$, we have:
\begin{align*}
&Y^{x,u}_0(t)+\left \langle Y^{x,u}_1(t),h_1^k \right\rangle_{L^2}\\=& \ x_0+ \langle x_1,h_1^k \rangle_{L^2}+ \int_0^t\left \langle Y^{x,u}_1(s),\frac{dh_1^k}{d\xi} \right\rangle_{L^2} ds + \int_0^t b_0(Y^{x,u}_0(s),u(s))ds\\
& +\int_0^t \left\langle a_1 Y^{x,u}_0(s)+p_1 u(s) , h_1^k\right\rangle_{L^2}  ds + \int_0^t \sigma_0(Y^{x,u}_0(s),u(s))dW(s).
\end{align*}
Note that, for every $t \geq 0$, on a has $Y_1(t)(\cdot) \in {C}([-d,0];\R^{n})$, since $Y_1(t)(\cdot)$ is the sum of the convolutions of $L^2$-functions.
Then, taking  $k\to + \infty$ in the equation above, we get
\begin{align*}
Y^{x,u}_0(t)= x_0+  \int_0^t Y^{x,u}_1(s)(0)ds + \int_0^t b_0(Y^{x,u}_0(s),u(s))ds+ \int_0^t \sigma_0(Y^{x,u}_0(s),u(s))dW(s).
\end{align*}
By \eqref{eq:proof_gozzi_marinelli_Y1_rewritten_2}, with $\xi=0$ we have:
$$
 Y^{x,u}_{1}(s)(0) =\int_{-d}^{0} a_{1}(\xi) \tilde Y^{x,u}_{0}(s+\xi) d \xi+\int_{-d}^{0} p_{1}(\xi) u(s+\xi) d \xi,
$$
so that, for every $t \geq 0$,
\begin{align*}
Y^{x,u}_0(t)= &x_0+  \int_0^t \left[ \int_{-d}^{0} a_{1}(\xi) \tilde Y^{x,u}_{0}(s+\xi) d \xi +\int_{-d}^{0} p_{1}(\xi) u(s+\xi) d \xi  \right] ds + \int_0^t b_0(Y^{x,u}_0(s),u(s))ds\\
&+ \int_0^t \sigma_0(Y^{x,u}_0(s),u(s))dW(s).
\end{align*}
%i.e. 
%\begin{align*}
% Y^{x,u(\cdot)}_0(t)= &\eta_0+  \int_0^t \left[ \int_{-d}^{0} a_{1}(\xi) \tilde{Y}_{0}^{x,u(\cdot)}(s+\xi) d \xi +\int_{-d}^{0} p_{1}(\xi) u(s+\xi) d \xi  \right] ds + \int_0^t b_0(Y^{x,u(\cdot)}_0(s),u(s))ds\\
%&+ \int_0^t \sigma_0(Y^{x,u(\cdot)}_0(s),u(s))dW(s).
%\end{align*}
Recalling the definition of $\tilde Y_0^{x,u(\cdot)}$,
 this says that $\tilde Y_{0}^{x,u(\cdot)}$ satisfies \eqref{eq:SDDE}, so we conclude.
 \end{enumerate}
\end{proof}
\subsection{Objective functional} Using Theorem \ref{th:equiv}, we can give a Markovian
reformulation on the Hilbert space $X$ of the optimal control problem. 

We present such result for an optimal control problem with the more general state equation \eqref{eq:SDDE}. Consider the functional $\mathcal J$ defined by \eqref{eq:obj-origbis} with $y^{\eta,u(\cdot),\delta(\cdot)}$ being the solution of  \eqref{eq:SDDE} (in place of \eqref{eq:SDDE_2}).
Denoting by $Y^{x,u}$ a mild solution of
(\ref{eq:abstract}) for general initial datum $x \in X$ and control $u(\cdot) \in \mathcal{U}$ and introducing the functional 
\begin{equation}
\label{ex1jbisbis}
J(x;u(\cdot)) :=
\mathbb E \left[ \int_0^{\infty} e^{-\rho t}L(Y^{x,u}(t),u(t))\,dt \right],
\end{equation}
where
$$
L : X \times U\to\erre, \ \ \ 
L(x,u) = l(x_0,u),
$$
the original functional $\mathcal J$  and $J$ are related through
$$
\mathcal J(\eta,\delta;u(\cdot))=J(M(\eta,\delta);u(\cdot)). 
$$
We then consider the problem of optimizing $J$ under \eqref{eq:abstract}
and  define the value function $V$ for this problem:
\begin{equation}\label{valuing}
V(x) = \inf_{u\in\mathcal{U}} J(x;u(\cdot)).
\end{equation}
For what we said,
an optimal control  $u^*(\cdot)\in\mathcal{U}$ for the functional $J(x;\cdot)$ with $x=M(\eta,\delta)$ is also optimal for $\mathcal J(\eta,\delta;\cdot)$.
Hence, from now on, we focus on the optimization problem \eqref{valuing}.
\section{$B$-continuity}\label{sec:B_continuity}
In order to get  $B$-continuity of the value function $V$, needed  to employ the theory of viscosity solutions in infinite dimension,  we consider the simpler state equation \eqref{eq:SDDE_2}. In this case, the state equation \eqref{eq:SDDE_2} can be rewritten in infinite dimension as
\begin{equation}
\label{eq:abstract2}
\begin{cases}
dY(t) = [\mathcal A Y(t)+f(u(t))]dt + \sigma(u(t))\,dW(t), \\
Y(0) = x=M(\eta,\delta) \in X,
\end{cases}
\end{equation}
where 
$$\mathcal A : D(\mathcal A)=D(A) \subset X \to X, \ \ \ \ \ \mathcal A x =Ax+\begin{bmatrix}
 a_0 \\
  a_1
\end{bmatrix}x_{0}=\begin{bmatrix}
a_0 x_0+ x_1(0)\\
a_1 x_0   -x_1'
\end{bmatrix},
$$
\begin{align*}
f\colon  U \to X, \ \ \  f(u)=\begin{bmatrix}
b_0(u)\\
p_1 u
\end{bmatrix},
\end{align*}
and
\begin{align*}
\sigma \colon U \to \mathcal{L}(\mathbb{R}^q,X), \ \ \ \ 
\sigma(u)w=\begin{bmatrix}
\sigma_0(u)w\\
0
\end{bmatrix}.
\end{align*}
Indeed, since $\mathcal A$ is the sum of $A$  with a bounded linear operator, by \cite[Corollary 1.7]{engel_nagel}, we have that \eqref{eq:abstract} with specifications
 \begin{align} \label{eq:coefficient_restriction}
b_0(x_0,u)=a_0x_0+b_0(u), \ \ \ 
 \sigma_0(x_0,u)=\sigma_0(u),
\end{align}
and \eqref{eq:abstract2} are equivalent and have the same (unique) mild solution $Y^{x,u}(t)$.  %Also the original functional $\mathcal J$, defined by \eqref{eq:obj-origbis},  is rewritten as in \eqref{ex1jbisbis}, where here $Y^{x,u}(t)$ is a mild solution to \eqref{eq:abstract2}.

\subsection{Reformulation with a maximal dissipative operator}
 The aim of this subsection is to rewrite  \eqref{eq:abstract2}  with a maximal dissipative operator $\tilde{\mathcal A}$ in place of $\mathcal A$. The need for that is that we want  to use the viscosity solutions theory in infinite dimension to treat the HJB equation associated to the optimal control problem, which requires for the comparison theorem the presence of a maximal dissipative operator in the equation (see \cite[Chapter 3]{fgs_book}). The operator $\tilde{\mathcal A}$ is constructed by means of a suitable  shift of the operator $\mathcal A$. 
% In the following, we denote the identity operator on a vector space $V$  by $I_{V}$.
  \begin{prop} \label{prop:Amaxdiss}
 There exists $\mu_0 >0$ such that  $\tilde{\mathcal A}_\mu \colon D(\tilde {\mathcal A}_\mu) =D(A), \subset X \to X$ defined by  
\begin{align*}
\tilde{\mathcal A}_\mu x:=\mathcal A x-\mu x=\begin{bmatrix}
a_0x_{0}-\mu x_0+ x_1(0)\\
a_1 x_0 -\mu x_1  -x_1'
\end{bmatrix} = Ax + \begin{bmatrix}
a_0x_0-\mu  x_0\\
a_1 x_0 -\mu x_1  
\end{bmatrix}, \quad x \in D(\tilde {\mathcal A}_\mu)
\end{align*}
is  maximal dissipative for every $\mu \geq \mu_0$.
\end{prop}
\begin{proof}
\emph{Step 1.} 
We prove that $\tilde{\mathcal A}_\mu$ is dissipative.
Let 
$x \in D(A)=\left\{ x \in X:
       x_1 \in W^{1,2}([-d,0]; \mathbb{R}^n), \ x_1(-d)=0\right\},$ then
\begin{align*}
\langle
\mathcal A x,x
\rangle_X
& =(a_0 x_0+x_1(0) ) \cdot x_0+  \int_{-d}^0  a_1(\xi)x_0 \cdot x_1(\xi)d\xi - \int_{-d}^0 x_1'(\xi) \cdot x_1(\xi)d\xi\\
& =(a_0 x_0+x_1(0) ) \cdot x_0 + x_0 \cdot \int_{-d}^0  a_1(\xi)^T x_1(\xi)d\xi
- \frac12 x_1(0)^2 + \frac12 x_1(-d)^2\\
& =(a_0 x_0+x_1(0) ) \cdot x_0 + x_0 \cdot \int_{-d}^0  a_1(\xi)^T x_1(\xi)d\xi
- \frac12 x_1(0)^2 \\
& \leq a_0 x_0 \cdot x_0 + \frac 12|x_0|^2 + x_0 \cdot \int_{-d}^0  a_1(\xi)^T x_1(\xi)d\xi\\
& \leq (|a_0|+1) |x_0|^2 +\frac{1}{2}|a_1|^2_{L^2_{-d}}|x_1|^2_{L^2_{-d}} \leq \mu_0 |x|^2,
\end{align*}
where $\mu_0:=\max \left \{ |a_0|+1,\frac{1}{2} |a_1|^2_{L^2}\right\}$.
This implies that $\tilde{\mathcal A}_\mu$ is dissipative for every $\mu \geq \mu_0$.
\medskip

\emph{Step 2.}
We prove that $\tilde{\mathcal A}_\mu$ is maximal for $\mu \geq \mu_0$. For that we show that there exists $\lambda>0$ such that  $R(\lambda I  - \tilde{\mathcal A}_\mu)=X$. This means that, for every $y\in X$, there exists a solution $x \in D(\tilde{\mathcal A}_\mu)$ to the equation
\begin{equation}\label{eq:lambda x -Cx=y}
\lambda x -\tilde{\mathcal A}_\mu x=y.
\end{equation}
Fix $\lambda>0$. Let then $y \in X$. The equation above rewrites as
$$
\begin{cases}
(\lambda+\mu) x_0-a_0 x_0 - x_1 (0) =y_0,
\vspace{2mm}\\
\begin{cases}
(\lambda+\mu) x_1(\xi) - a_1(\xi) x_0 + x_1'(\xi) =y_1(\xi), \quad \xi \in [-d,0]\\
x_1(-d)=0.
\end{cases}
\end{cases}
$$
This system is uniquely solvable. Indeed,  from $(\lambda+\mu)  x_0-a_0 x_0 - x_1 (0) =y_0$, we get uniquely
\begin{equation}\label{eq:proof_eq_x_0_maximality}
x_1(0)=(\lambda+\mu)  x_0-a_0 x_0 - y_0.
\end{equation}
On the other hand, 
$$\begin{cases}
(\lambda+\mu)  x_1(\xi) - a_1(\xi) x_0 + x_1'(\xi) =y_1(\xi), \quad \xi \in [-d,0]\\
x_1(-d)=0.
\end{cases}
$$ 
yields the unique solution
$$
x_1(\xi)= \int_{-d}^\xi e^{-(\lambda+\mu)  (\xi- r)}(y_1 (r) + a_1(r) x_0)dr.
$$
Taking $\xi =0$ in this equality and equating with \eqref{eq:proof_eq_x_0_maximality}, we have:
$$
(\lambda+\mu)  x_0-a_0 x_0 - y_0 =  \int_{-d}^0 e^{(\lambda+\mu)  r} y_1(r)dr+\int_{-d}^0 e^{(\lambda+\mu)  r} a_1(r)dr x_0.
$$
Then, for $\mu \geq \mu_0$
$$
x_0= \left [ (\lambda+\mu) I_{\R^{n}} -a_0-\int_{-d}^0 e^{(\lambda+\mu) r} a_1(r)dr \right]^{-1}\left [ y_0 + \int_{-d}^0 e^{(\lambda+\mu)r} y_1(r)dr \right ].
$$
Therefore, we have proved that, for every $\mu \geq \mu_0$ and $y \in X$, there exists  a unique solution $x \in D(\tilde {\mathcal A}_\mu)$ to \eqref{eq:lambda x -Cx=y} given by
\begin{align}\label{eq:proof_solution_maximally_diss}
x=\begin{bmatrix}
x_0\\
x_1
\end{bmatrix}=
\begin{bmatrix}
\left [ (\lambda+\mu) I_{\R^{n}} -a_0-\int_{-d}^0 e^{(\lambda+\mu) r} a_1(r)dr \right]^{-1}\left [ y_0 + \int_{-d}^0 e^{(\lambda+\mu) r} y_1(r)dr \right ]\\\\
 \int_{-d}^\cdot  e^{-(\lambda+\mu) (\cdot- r)}(y_1 (r) + a_1(r) x_0)dr
\end{bmatrix}.
\end{align}
The claim follows.
\end{proof}
Now, we fix $\mu  > \mu_0$ and denote 
$$\tilde{\mathcal A}:=\tilde{\mathcal A}_\mu=\mathcal A - \mu I.$$ 
We 	may rewrite  SDE \eqref{eq:abstract2} as
\begin{equation}
\label{eq:abstract_dissipative_operator}
\begin{cases}
dY(t) = \left[ \tilde{\mathcal A} Y(t)+\tilde b(Y(t),u(t)) \right]dt +  \sigma(u(t))\,dW(t),  \\[8pt]
Y(0) = x \in X,
\end{cases}
\end{equation}
where  
\begin{align*}
\tilde b \colon X \times U \to X, \ \ \ \ \tilde b(x,u) & = \mu x+ f(u)=\begin{bmatrix}
\mu x_0+b_0(u)\\
\mu x_1+p_1 u 
\end{bmatrix},\quad x=(x_0,x_1) \in X, \ u \in U.
\end{align*}
%and, as before, $\sigma \colon U \to \mathcal{L}(\mathbb{R}^q,X)$ 
%\begin{align*}
%\sigma(u)w=\begin{bmatrix}
%\sigma_0(u)w\\
%0
%\end{bmatrix},
%\quad x=(x_0,x_1) \in X, u \in U, w \in \mathbb{R}^q 
%\end{align*}
Similarly to before, since the operator $\tilde{\mathcal A}$ is the sum of $A$  with a bounded operator, by \cite[Corollary 1.7]{engel_nagel}, we have that \eqref{eq:abstract_dissipative_operator} is equivalent to  \eqref{eq:abstract2} (and so also to \eqref{eq:abstract} with the condition \eqref{eq:coefficient_restriction}) and  all them have the same (unique) mild solution, that in terms of $\tilde{\mathcal A}$ writes as  
\begin{equation}\label{eq:mild}
Y(t) =e^{{\tilde{\mathcal A} t}}x+\int_{0}^{t}  e^{\tilde{\mathcal A} (t-s)} \tilde b(Y(s),u(s)) d s+\int_{0}^{t} e^{\tilde{\mathcal A}(t-s)}\sigma(u(s))d W(s).
\end{equation}

\subsection{Weak $B$-condition}\label{sec:operator_B}
In this subsection, we recall the concept of weak $B$-condition for  the  maximal dissipative operator  $\tilde{\mathcal A}$ and introduce an operator $B$ satisfying it. 
This concept is  fundamental in the theory of viscosity solutions in Hilbert spaces, see \cite[Chapter 3]{fgs_book}, which will be used in this paper.

\begin{definition}\label{def:weakB}(\cite[Definition 3.9]{fgs_book})
We say that $B \in \mathcal{L}(X)$ satisfies the weak $B$-condition for $\tilde{\mathcal A}$ if the following hold:
\begin{enumerate}[(i)]
\item $B$ is strictly positive, i.e. $\langle Bx,x \rangle_{X} >0$ for every $x \neq 0$;
\item $B$ is self-adjoint;
\item $\tilde{C}^* B \in \mathcal{L}(X)$;
\item There exists $c_0 \geq 0$ such that
$$\langle \tilde{\mathcal A}^* Bx,x \rangle_{X} \leq c_0 \langle Bx,x \rangle_{X}, \quad \forall x \in X.$$
\end{enumerate}
\end{definition}
Let $\tilde{\mathcal A}^{-1}$ be the inverse of the operator $\tilde{\mathcal A}$. Its explicit expression can be derived as in the proof of Proposition \ref{prop:Amaxdiss}: 
%by considering \eqref{eq:lambda x -Cx=y} with $\lambda=0$  and $y=-z$:
% Note that $\lambda=0$ does not invalidate the calculations in the proof of Proposition \ref{prop:Amaxdiss}. Indeed, let $\mu=\mu_0$ in Proposition \ref{prop:Amaxdiss}. Then, by the proof of the Lumer-Philips theorem \cite[Theorem 4.3]{pazy}, we have that the resolvent operator of $\tilde C_{\mu_0}$ $R^{\tilde C_{\mu_0}}_{\lambda}=(\lambda I-\tilde C_{\mu_0})^{-1} \in \mathcal L(X)$ for every $\lambda>0$. Then this is true for $\lambda = \mu - \mu_0>0$ (recall that we fixed $\mu >\mu_0$ in \eqref{eq:abstract_dissipative_operator}). 
%Now by definition we have $\tilde C=(\mu_0-\mu)I+\tilde C_{\mu_0}$ so that
%$(-\tilde C)^{-1}= ((\mu-\mu_0) I-\tilde C_{\mu_0})^{-1}=R^{\tilde C_{\mu_0}}_{\mu-\mu_0} \in \mathcal L(X)$ which implies that $\tilde C$ is invertible with  $\tilde C^{-1} \in \mathcal L(X)$. Its explicit expression can then be derived by the proof of Proposition \ref{prop:Amaxdiss}, by considering \eqref{eq:lambda x -Cx=y} with $\lambda=0$  and $y=-z$. The solution is \eqref{eq:proof_solution_maximally_diss} with $\lambda=0$ and $y=-z$, i.e.
\begin{equation}\label{eq:tilde_A_inverse}
 \tilde{\mathcal A}^{-1} z
 =
 \begin{bmatrix}
-\left [ \mu I_{\R^{n}} -a_0-\int_{-d}^0 e^{\mu r} a_1(r)dr \right]^{-1}\left [ z_0 + \int_{-d}^0 e^{\mu r} z_1(r)dr \right ]\\\\
 \int_{-d}^\cdot  e^{-\mu (\cdot- r)}(-z_1 (r) + a_1(r) z_0)dr
\end{bmatrix},
\   z=(z_0,z_1) \in X.
\end{equation}
Notice that $\tilde{\mathcal A}^{-1} \in \mathcal{L}(X)$  and it is compact.
Define now  the compact operator
\begin{equation}\label{eq:def_B}
B:=(\tilde{\mathcal A}^{-1})^*\tilde{\mathcal A}^{-1}=(\tilde{\mathcal A}^*)^{-1}\tilde{\mathcal A}^{-1}.
\end{equation}
We are going to show that $B$ satisfies the weak-$B$ condition for $\tilde{\mathcal A}$.
\begin{prop}\label{prop:weak_B}
$B$ defined in \eqref{eq:def_B} satisfies the weak-$B$ condition for $\tilde{\mathcal A}$ with $c_0=0$.
\end{prop}
\begin{proof}
Clearly, $B \in \mathcal{L}(X)$, $\tilde{\mathcal A}^* B= \tilde{\mathcal A}^{-1} \in \mathcal{L}(X)$, and $B$ is self adjoint. Moreover, $B$ is strictly positive; indeed 
\begin{align*}
\langle  Bx,x \rangle_X= \langle \tilde{\mathcal A}^{-1}x,\tilde{\mathcal A}^{-1} \rangle_X= |\tilde{\mathcal A}^{-1}x|^2 \geq 0, \ \ \ \ \forall x\in X,
\end{align*}
 and it is easy to check that, whenever $x \neq 0$, we have that $ |\tilde{\mathcal A}^{-1}x|>0$.
Finally, by  dissipativity of $\tilde{\mathcal A}$, we have:
\begin{align*}
\langle \tilde{\mathcal A}^*  Bx,x \rangle_X = \langle \tilde{\mathcal A}^{-1}x ,x \rangle_X = \langle y, \tilde{\mathcal A}y \rangle_X \leq 0
\end{align*}
with $y= \tilde{\mathcal A}^{-1}x$.
\end{proof} 
Observe that 
\begin{align}\label{eq:representation_B}
B x= 
  \begin{bmatrix}
    B_{00} & B_{01}  \\
   B_{10} & B_{11}
  \end{bmatrix} 
   \begin{bmatrix}
   x_0 \\
  x_1
  \end{bmatrix}, \ \ \ \ x=(x_0,x_1)\in X. 
\end{align}
By strict positivity of $B$, we have that $B_{00}$    and $B_{11}$ are strictly positive. We introduce the $|\cdot |_{-1}$-norm on $X$ by
\begin{align}\label{eq:properties_norm_-1}
\nonumber|x|_{-1}^2&:=\langle B^{1/2}x, B^{1/2}x\rangle_{X}= \langle Bx,x\rangle_{X}\\&=\langle (\tilde{\mathcal A}^{-1})^*\tilde{\mathcal A}^{-1}x,x\rangle_{X}=\langle \tilde{\mathcal A}^{-1}x,\tilde{\mathcal A}^{-1}x\rangle_{X}= |\tilde{\mathcal A}^{-1}x|^{2}_{X} \quad \forall x \in X. 
\end{align}
We define 
%Moreover, since $B$ is strictly positive, compact, and self-adjoint, we can define the strictly positive, self-adjoint, and  compact operator $B^{\alpha}$ for every $\alpha \in \mathbb{R}$ (SF: COME?). 
%
%For $\alpha\in \R$, let us define the inner product in $X$:
%\begin{equation}\label{Bscalar}
%\langle x,y \rangle_{-\alpha}=\langle B^{\alpha / 2} x, B^{\alpha / 2} y\rangle_{X}, \ \ \ \ x,y\in X.
%\end{equation}
%We define a family of  spaces $\left\{X_{\alpha}\right\}_{\alpha \in \mathbb{R}}$  as follows:
%\begin{enumerate}[(i)]
%\item If $\alpha\geq 0$, then  $X_{\alpha}:=\mathcal{R}\left(B^{\alpha / 2}\right)$;
%\item If $\alpha<0$, then
%
$$X_{-1}:= \ \mbox{  the completion of $X$ under} \ |\cdot|_{-1},$$
which is a Hilbert space endowed with the inner product
$$
\langle x,y\rangle_{-1}:=\langle B^{1/2}x,B^{1/2}y\rangle_{X}= \langle Bx,y\rangle_{X}= \langle \tilde{\mathcal A}^{-1}x,\tilde{\mathcal A}^{-1}y\rangle_{X}.
$$
Notice that $|x|_{-1}\leq |\tilde{\mathcal A}^{-1}|_{\mathcal L(X)}|x|_X$; in particular,  we have $(X,|\cdot|) \hookrightarrow (X_{-1},|\cdot|_{-1})$. 
Strict positivity of $B$ ensures that the operator $B^{1 /2}$  can be extended to an isometry 
$$B^{1 /2} \colon (X_{- 1},|\cdot|_{-1}) \to (X,|\cdot|_{X}).$$

%\end{enumerate}
%
%
%We have the following classical result, whose proof can be found in (SF: CITAZIONE PIU' PRECISA) \cite{LY}:
%
%\begin{prop}\label{prop:spaces_X_alpha} 
%There exists a family of Hilbert spaces $\left\{X_{\alpha}\right\}_{\alpha \in \mathbb{R}}$ such that:
% $$\quad X_{0}=X, \quad X_{\alpha} \hookrightarrow X_{\beta}, \quad \forall \alpha \geq \beta,$$
% $$\quad\left(X_{-\alpha}\right)^{*} := \textit{dual of} \ X_{-\alpha}=X_{\alpha}, \quad \forall \alpha \in \mathbb{R},$$
%$$X_{\alpha}=\mathcal{R}\left(B^{\alpha / 2}\right), \quad \alpha \geq 0,$$ endowed by the scalar product $$ \langle x,y \rangle_{\alpha} =\langle B^{-\alpha / 2} x, B^{-\alpha / 2} y\rangle_{X}$$
%and such that
%$$X_{-\alpha}= \textit{the completion of X under the norm induced by the scalar product } $$ 
%$$\langle x,y \rangle_{-\alpha}=\langle B^{\alpha / 2} x, B^{\alpha / 2} y\rangle_{X}$$
%\end{prop}
%In particular we have
%$$ X_{2} \hookrightarrow X_{1} \hookrightarrow X \hookrightarrow X_{-1} \hookrightarrow X_{-2}$$
By \eqref{eq:properties_norm_-1} and  an application of \cite[Proposition B.1]{DZ14},  we have  
$\mbox{Range}(B^{1/2})=\mbox{Range}((\tilde{\mathcal A}^{-1})^*)$. Since $\mbox{Range}((\tilde{\mathcal A}^{-1})^*)=D(\tilde{\mathcal A}^*)$, we have 
\begin{align}\label{eq:R(B^1/2)=D(A^*)}
\mbox{Range}\big(B^{1/2}\big)=D(\tilde{\mathcal A}^*).
\end{align}
By \eqref{eq:R(B^1/2)=D(A^*)}, the operator $\tilde{\mathcal A}^* B^{1/2}$ is well defined on the whole space $X$. Moreover, since $\tilde{\mathcal A}^*$ is closed and $B^{1/2}\in\mathcal{L}(X)$, $\tilde{\mathcal A}^* B^{1/2}$ is a closed operator. Thus, by the closed graph theorem, we have 
\begin{equation}\label{eq:A*B^1/2}
\tilde{\mathcal A}^* B^{1/2} \in \mathcal{L}(X).
\end{equation}
\begin{remark}
In the infinite dimensional theory of viscosity solutions it is only required that $\tilde{\mathcal A}^*B \in \mathcal{L}(X)$ (condition $(iii)$ of Definition \ref{def:weakB}). Such an operator can be constructed for any maximal dissipative operator $\tilde{\mathcal A}$ (see, e.g.,  \cite[Theorem 3.11]{fgs_book}). Similarly to \cite{defeo_federico_swiech} in the case of the present paper, in addition, we also have $\tilde{\mathcal A}^* B^{1/2} \in \mathcal{L}(X)$. Such stronger condition 
 may be helpful in order to get differentiability properties of the value function (see \cite[Proof of Theorem 6.5]{defeo_federico_swiech}) or in order to construct optimal feedback laws  (see  \cite{deFeoSwiech}).
\end{remark}
Since $B$ is a compact, self-adjoint and strictly positive operator on $X$, by the spectral theorem $B$ admits a set  of  eigenvalues $\{\lambda_i \}_{i \in \mathbb{N}} \subset (0, +\infty)$ such that $\lambda_{i}\to 0^{+}$ and a corresponding set $\{f_i \}_{i \in \mathbb{N}} \subset X$  of  eigenvectors  forming an orthonormal basis of $X$. By taking $\{e_i \}_{i \in \mathbb{N}}$ defined by $e_i:=\frac{1 } {\sqrt \lambda_i} f_i$, we then get   an orthonormal basis of $X_{-1}$. 
We set
$X^N := \mbox{Span} \{f_1,...f_N\}= \mbox{Span} \{e_1,...e_N\}$ for  $N\geq 1$, and let 
$P_N \colon X \to X$
 be the orthogonal projection onto $X_N$ and 
$Q_N := I - P_N$. 
Since $\{e_i \}_{i \in \mathbb{N}}$ is an orthogonal basis of $X_{-1}$, the projections $P_N,Q_N$ extend to orthogonal projections in $X_{-1}$ and we will use the same symbols to denote them.
% Moreover we claim that $P_N$ is an equivalent projection also in the $X$-topology i.e. if you consider $\bar P_N \colon X \to X$ the %orthogonal projection on $X_N$ in the $X$-topology  then $P_N x=\bar P_N x$ for every $x \in X$. Indeed let $x \in X$ and decompose it %as $x = \sum_{i \in \mathbb{ N}} x_i f_i$ for some $\{x_i\}_i \subset \mathbb{R}$, then
% \begin{align*}
%& \bar  P_N x= \sum_{i \in \mathbb N} x_i \bar P_N f_i= \sum_{i=1}^N x_i f_i\\
%& P_N x =  \sum_{i \in \mathbb N} x_i P_N f_i = \sum_{i \in \mathbb N} x_i \sqrt{\lambda_i} P_N e_i = \sum_{i=1}^N  x_i \sqrt{\lambda_i} %P_N e_i =\sum_{i=1}^N x_i f_i
% \end{align*}
 %so we have the claim.\\
%Then we will see $P_N, Q_N$ also as operators on $X$, i.e. $P_N, Q_N \colon X \to X$, which are bounded self-adjoint, nilpotent operator %being a projection operator on $X$.\\
We notice that 
\begin{equation}\label{eq:BP_BQ}
B P_N =P_N B, \quad B Q_N =Q_N B.
\end{equation}
%Indeed as $P_N$ is a projection operator in $X_{-1}$ we have for every $x,y \in X$
%\begin{align*}
%\langle B P_N x,y \rangle = \langle P_N x,y \rangle_{-1} = \langle x,P_N y \rangle_{-1} =\langle x,B P_N y \rangle
%\end{align*}
%so that $BP_N = (BP_N)^*=P_N^*B^*=P_N B$ and then also $B Q_N =Q_NB$ follows.\\
Therefore, since $|BQ_N|_{\mathcal L(X)}=|Q_NB|_{\mathcal L(X)}$ and $B$ is compact, we get
\begin{equation}\label{eq:norm_BQ_N_to_zero}
\lim_{N \to \infty }|BQ_N|_{\mathcal L(X)} =0.
\end{equation}
\begin{remark}\label{rem:x_0_leq_x_-1}

\begin{enumerate}[(i)]
\item We remark that the following inequality 
\begin{equation}\label{eq:x_0_leq_x_-1}
|x_0| \leq C_R |x|_{-1} \quad \forall x=(x_0,x_1) \in B_R \subset X,
\end{equation} 
 (which was proven in \cite{defeo_federico_swiech}) here is false. We first claim that such inequality does not hold on unbounded sets of $X$. Indeed,  we  provide the following counter example which is similar to the one in \cite{tacconi}. Let $n=1$ and consider
$$
x^{N}=\left(x_{0}^{N}, x_{1}^{N}\right), \quad x_{0}^{N}=1, \quad x_{1}^{N}=-N \mathbf{1}_{[-1 / N, 0]}, \quad N\geq 1 \text {. }
$$
For $N$ big enough $-1 / N>-d$, then we have
$$
\left|x^{N}\right|_{-1}=\left|\tilde {\mathcal A}^{-1} x^{N}\right|=0+\int_{-\frac{1}{N}}^{0}\left|\int_{-\frac{1}{N}}^{\xi} n d s\right|^{2} d \xi=\int_{-\frac{1}{N}}^{0} N^{2}\left(\xi+\frac{1}{N}\right)^{2} d \xi=\frac{1}{3 N} \longrightarrow 0 .
$$
Therefore, we have $\left|x_{0}^{N}\right|=1$ and $\left |x^{N}\right|_{-1} \rightarrow 0$.\\
Now note that if the inequality \eqref{eq:x_0_leq_x_-1} holded on bounded sets $B_R$ then it would hold on the whole $X$ with $C_R=C>0$ which is a contradiction. Indeed assume the inequality holds on bounded sets $B_R$. Then it is true for every $x=(x_0,x_1)$ with $|x|\leq 1$, i.e. we have
\begin{equation}\label{eq:x_0_leq_x_-1_ball_1}
|x_0|\leq C|x|_{-1}, \quad |x|\leq 1
\end{equation}
Now let any $y=(y_0,y_1) \in X$, $y \neq 0$, define $x=y/|y|=(y_0/|y|,y_1/|y|)$, $|x|=1$ so that \eqref{eq:x_0_leq_x_-1_ball_1} holds and then by multiplying the inequality by $|y|$ we have
\begin{equation*}
|y_0|\leq C|y|_{-1},
\end{equation*}
which is a contradiction.
\item We point out that the fact that here \eqref{eq:x_0_leq_x_-1_ball_1} is false is not in contradiction with  \cite{defeo_federico_swiech}, where such inequality was shown to be true. Indeed we recall that the operators $\tilde {\mathcal A}$ here and in \cite{defeo_federico_swiech} are not the same: $\tilde {\mathcal A}$ here is the adjoint (up to  some bounded perturbation) of the operator $\tilde {\mathcal A}$ in \cite{defeo_federico_swiech}.

In \cite{tacconi}, where delays appear in the control similarly to here, the inequality is then proved in the smaller space $\mathbb{R} \times W^{1,2}([-d,0])$, see \cite[remark 5.4]{tacconi} and this is enough since the optimal control problem is deterministic. This would not be possible here due to the presence of the Brownian motion, whose trajectories are not absolutely continuous. Indeed, it would not be possible to have $Y_1(t) \in W^{1,2}$ as  in \cite{tacconi}.
\end{enumerate}
\end{remark}
\subsection{$B$-continuity of the value function}\label{sec:estimates}
In this subsection, we prove some estimates for solutions of the state equation and on the cost functional in order to prove  the $B$-continuity of the  value function. 
\begin{lemma}\label{lemma:regularity_coefficients_theory}
There exists $C>0$ and a local modulus of continuity $\omega_R$ such that
\begin{align}
& |\tilde b(x, u)-\tilde b(y, u)|_X = \mu |x-y|_X \label{eq:b_lip}  \\
&  |\tilde b(x, u)-\tilde b(y, u)|_{-1}^2 =\langle \tilde b(x, u)-\tilde b(y, u), B(x-y)\rangle = \mu |x-y|_{-1}^{2}  \label{eq:b_B_property}  \\
& |\tilde b(x, u)|_X \leq C(1+|x|_X)  \label{eq:b_sublinear}  \\
& |\sigma(u)|_{\mathcal{L}_{2}(\mathbb R^q, X)} \leq C \label{eq:G_bounded}  \\
&|L(x, u)-L(y, u)| \leq \omega_R\left(|x-y|_{-1}\right) \label{eq:L_unif_cont_norm_B}  \\
& |L(x, u)| \leq C\left(1+|x|_X^{m}\right)   \label{eq:L_growth} 
\end{align}
for every $x,y \in X, u \in U$, $R>0$.\\
Moreover
\begin{equation}\label{eq:trace_goes_to_zero}
\lim_{N \to \infty} \sup_{u \in U} \operatorname{Tr}\left[\sigma(u) \sigma(u)^{*} B Q_{N}\right]=0 
\end{equation}
\end{lemma}
\begin{proof}
\begin{itemize}
\item \eqref{eq:b_lip} follows immdediately by the definition of $\tilde b(x,u)$.
\item Similarly for \eqref{eq:b_B_property}  we have:
\begin{align*}
 |\tilde b(x, u)-\tilde b(y, u)|_{-1}^2 =\langle \tilde b(x, u)- \tilde b(y, u), B(x-y)\rangle & = \mu \langle  x-y, B(x-y)\rangle= \mu |x-y|_{-1}^2.
\end{align*}
\item Recalling the definition of $\tilde b(x,u)$,   \eqref{eq:b_sublinear} follows the boundedness of $b_0(u)$ and the boundedness of $U.$
\item \eqref{eq:G_bounded} follows by the boundedness of $\sigma_0 \colon  U \to M^{n \times q}$ since
\begin{align*}
|\sigma(u)|_{\mathcal{L}_{2}(\mathbb R^q, X)}^2& =\sum_{i=1}^n |\sigma(u) r_i|^2=\sum_{i=1}^n |\sigma_0(u) r_i|^2\leq n |\sigma_0(u)|^2 \leq C
\end{align*}
where $\{r_i	\}_{i=1}^q$ is the classical  orthonormal basis of $\mathbb R^q$.
\item For \eqref{eq:L_unif_cont_norm_B} note that by \eqref{eq:regularity_l} the function $L(x,u)=l(x_0,u)$ is weakly sequentially continuous uniformly in $u \in U$ (as $x_0 \in \mathbb R^n$), i.e.
$$\sup_{u \in U} |L(y,u)-L(x,u)|=\sup_{u \in U} |l(y_0,u)-l(x_0,u)|\leq \omega_R(|y_0-x_0|)\to 0$$
for every $y \rightharpoonup x$.\\
 Then the inequality follows by application of \cite[Lemma 3.6 (iii)]{fgs_book} which can easily be extended to functions weakly sequentially continuous uniformly with respect to the control parameter.
\item \eqref{eq:L_growth} finally follows by the definition of $L$ and \eqref{eq:growth_l}.
\item We notice that by \cite[Appendix B]{fgs_book} and \eqref{eq:G_bounded}, we have:
\begin{small}
 \begin{align*}
|\operatorname{Tr}\left[\sigma(u) \sigma(u)^{*} B Q_{N}\right]|
& \leq  |\sigma(u) \sigma(u)^{*} BQ_N|_{\mathcal L_1(\mathbb R^q,X)} \leq  |\sigma(u) \sigma(u)^{*}|_{\mathcal L_1(\mathbb R^q,X)} |B Q_N|_{\mathcal L(X)} \\
& =    |\sigma(u)|_{\mathcal L_2(\mathbb R^q,X)}^2 |B Q_N|_{\mathcal L(X)}  \leq C |B Q_N|_{\mathcal L(X)}.
\end{align*}
\end{small}
Now taking the supremum over $u$ and letting $N \to \infty$ by \eqref{eq:norm_BQ_N_to_zero} we have \eqref{eq:trace_goes_to_zero}.
\end{itemize}
\end{proof}
\begin{remark}
We remark that the fact that \eqref{eq:x_0_leq_x_-1}  is false (see Remark \ref{rem:x_0_leq_x_-1}) was the reason that lead us to consider the simpler state equation \eqref{eq:SDDE_2}, in place of the more general state equation \eqref{eq:SDDE},  in order to use the approach via viscosity solutions: regarding $b_0,$ for instance,  if \eqref{eq:x_0_leq_x_-1} were true, we could have considered a  more general $b_0$ of the form $b_0(x_0,u),$ satisfying
$$|b_0(x_0,u)-b_0(y_0,u)| \leq C |x_0-y_0|.$$ 
 Indeed, using \eqref{eq:x_0_leq_x_-1}, we would have had 
$$|b_0(x_0,u)-b_0(y_0,u)| \leq  C_R|x-y|_{-1},$$ 
from which we would have proved 
$$ \langle \tilde b(x, u)-\tilde b(y, u), B(x-y)\rangle \leq C_R |x-y|_{-1}^{2} . $$
Similarly we could have allowed $\sigma_0(x_0,u)$ to depend also on $x_0$. This would have been enough in order to prove the $B$-continuity of $V$ and the validity of the hypotheses of the comparison theorem \cite[Theorem 3.56]{fgs_book} when the state equation is of the form \eqref{eq:SDDE}. This  approach was used in \cite{defeo_federico_swiech}, where this inequality was proved to be true. 
\end{remark}
We recall \cite[Proposition 3.24]{fgs_book}. Set 
\[ \rho_{0}:=
\begin{cases}
0,              & \text { if } m =0,\\
C m+\frac{1}{2} C^{2} m,  &\text { if } 0<m<2, \\
C m+\frac{1}{2} C^{2} m(m-1),        & \text { if } m \geq 2, 
\end{cases}
\]
where $C$ is the constant appearing in \eqref{eq:b_sublinear} and \eqref{eq:G_bounded}, and $m$ is the constant from Assumption \ref{hp:cost} and \eqref{eq:L_growth}.

\begin{prop}\label{prop:expect_Y(s)} (\cite[Proposition 3.24]{fgs_book})
Let Assumption \ref{hp:state} holds and let $\lambda>\rho_0$.
Let $Y(t)$ be the mild solution of \eqref{eq:abstract_dissipative_operator} with initial datum $x \in X$ and control $u(\cdot) \in \mathcal U$. Then, there exists $C_\lambda>0$ such that
$$
\mathbb{E}\left[|Y(t)|_X^m\right] \leq C_\lambda\left(1+|x|_X^m\right)e^{\lambda t}, \quad \forall  t \geq 0.
$$
\end{prop}
We need the following assumption.
\begin{assumption}\label{hp:discount}
$\rho >\rho_{0}.$
\end{assumption}

\begin{prop}\label{prop:growth_V}
Let Assumptions \ref{hp:state}, \ref{hp:cost},  and \ref{hp:discount} hold. There exists $C>0$ such that
$$
|J(x;u(\cdot))|\leq C(1+|x|_X^m) \quad \forall x \in X,  \ \forall u(\cdot) \in \mathcal{U}.
$$
Hence, 
$$|V(x)|\leq C(1+|x|_X^m), \ \ \ \forall x \in X.$$
\end{prop}
\begin{proof}
By \eqref{eq:L_growth} and Proposition \ref{prop:expect_Y(s)} applied with $\lambda=(\rho+\rho_0)/2$, we have
\[
|J(x,u(\cdot))| \leq C \int_0^\infty e^{-\rho t} \mathbb E [ (1+|Y(t)|_X)^m] dt\leq C (1+|x|_X^m), \ \ \ \forall x \in X,  \ \forall u(\cdot) \in \mathcal{U}.
\]
The estimate on $V$  follows from this.
\end{proof}
Next, we show continuity properties of $V$. We recall first the notion of $B$-continuity (see \cite[Definition 3.4]{fgs_book})
\begin{definition}\label{def:B-semicontinuity}
 Let $B \in\mathcal{L}(X)$ be a strictly positive self-adjoint operator. 
 A function $u: X \rightarrow \mathbb{R}$ is said to be $B$-upper semicontinuous (respectively, $B$-lower semicontinuous) if, for any sequence $\left\{ x_{n}\right\}_{n \in \mathbb{N}}\subset X$ such that $ x_{n} \rightharpoonup x \in X$ and $B x_{n} \rightarrow B x$ as $n \rightarrow \infty$, we have
$$
\limsup _{n \rightarrow \infty} u\left( x_{n}\right) \leq u(x) \ \ \ \mbox{
(respectively, }
 \ \liminf _{n \rightarrow \infty} u\left( x_{n}\right) \geq u(x)).$$ 
  A function $u: X \rightarrow \mathbb{R}$ is said to be $B$-continuous if it is both $B$-upper semicontinuous and $B$-lower semicontinuous.
\end{definition}
We remark that, since our operator $B$  is compact,  in our case $B$-upper/lower semicontinuity is equivalent to the weak sequential upper/lower semicontinuity, respectively.

\begin{prop}\label{prop:V_continuous_norm_-1}
Let Assumptions \ref{hp:state}, \ref{hp:cost},  and \ref{hp:discount} hold.
For every $R>0,$ there exists a  modulus of continuity $\omega_R$ such that 
\begin{equation}\label{Va-1}
|V(x)-V(y)| \leq \omega_R(|x-y|_{-1}) , \ \ \ \forall x,y \in X, \ \mbox{s.t.} \ |x|_X, |y|_X \leq R.
\end{equation}
Hence $V$ is $B$-continuous and thus weakly sequentially continuous. 
\end{prop}
\begin{proof}
We prove the estimate
$$
|J(x,u)-J(y,u)| \leq \omega_R(|x-y|_{-1}) \quad \forall x,y \in X: \  |x|_X, |y|_X \leq R, \forall u(\cdot)\in \mathcal \mathcal{U}, 
$$
as in \cite[Proposition 3.73]{fgs_book}, since the assumptions of the latter are satisfied due to Lemma \ref{lemma:regularity_coefficients_theory}. Then, \eqref{Va-1} follows. 
As for the last claim, we observe that by  \eqref{Va-1} and by \cite[Lemma 3.6(iii)]{fgs_book}, $V$ is $B$-continuous in $X$. 
\end{proof}

We point out that $V$ may not be continuous  with respect to the $|\cdot|_{-1}$ norm in the whole $X$.

%Indeed, there might be  $\{x_N \} \subset X$ and  $x \in X$ such that $|x_N-x|_{-1}\to 0$ though $|x_N|_{X} \to 
%+\infty$. In this case we cannot conclude by the previous proposition that $|V(x_N)-V(x)|\to 0$.
%
%An example of such a sequence  is the following: 
%Fix $x \in X$ and set $y=\tilde C^{-1}x$. Take now $\{ y_N \} \subset D(\tilde C)$ such that $y_N \to y$ and define $x_N=\tilde C y_N$. By \eqref{eq:properties_norm_-1} $|x_N-x|_{-  1}=|\tilde C^{-1}(x_N-x)|=|y_N-y| \to 0$ but $|x_N|=|\tilde C y_N|$ is not necessarily bounded.\\

\section{The value function as unique viscosity solution to HJB equation}\label{sec:th_uiqueness}
In this section we prove that the value function $V$ is the unique viscosity solution of the infinite dimensional HJB equation.

Given $v \in C^1(X)$, we  denote by $D v(x)$ its Fr\'echet derivative at $x \in X$ and we write
\begin{align*}
Dv(x)=
  \begin{bmatrix}
    D_{x_0}v(x) \\
    D_{x_1}v(x)
  \end{bmatrix},
\end{align*}
where $D_{x_0}v(x), D_{x_1}v(x)$ are the partial Fr\'echet derivatives.
For $v \in C^2(X)$, we denote by $D^2 v(x)$ its second order Fr\'echet derivative at $x \in X$ which we will often write as
\[
D^2v(x)=
  \begin{bmatrix}
  D^2_{x_0^2}v(x) & D^2_{x_0x_1}v(x)  \\
    D^2_{x_1x_0}v(x)  & D^2_{x_1^2}v(x) 
  \end{bmatrix}.
\]
We define the Hamiltonian function  
 $H:X\times X\times S(X)\to \R$ by
\begin{small}
\begin{align*}
H(x,r,Z) &:= -\inf_{u\in U} \left \{ \cp{\tilde b(x,u)}{r}+  {1\over 2}\tr (\sigma(u)\sigma(u)^*Z) + L(x,u)  \right \} \\
 &= -\mu  x_0 \cdot r_0  -\mu  \langle x_1, r_1\rangle_{L^2} \\
&\quad -\inf_{u\in U} \left \{ b_0(u) \cdot r_0 + \langle  p_1 u,  r_1 \rangle_{L^2} +  {1\over 2} \tr \left [\sigma_0(u)\sigma_0(u)^T Z_{00} \right ] + l(x_0,u)  \right \} \\
 &= -\mu  x_0 \cdot r_0 -\mu  \langle x_1, r_1 \rangle_{L^2} \\
&\quad +\sup_{u\in U} \left \{ -b_0(u) \cdot r_0 - \langle  p_1 u,  r_1\rangle_{L^2} -  {1\over 2} \tr \left [\sigma_0(u)\sigma_0(u)^T Z_{00} \right ] - l(x_0,u)  \right \} \\
&=:\tilde H(x,r,Z_{00}),
\end{align*}
\end{small}
for every $x,r \in X, Z \in \mathcal S(X).$

By \cite[Theorem 3.75]{fgs_book} the Hamiltonian $H$ satisfies the following properties. 
\begin{lemma}\label{lemma:properties_H}
Let Assumptions \ref{hp:state} and \ref{hp:cost} hold.
\begin{enumerate}[(i)]
\item $H$ is uniformly continuous on bounded subsets of $X \times X \times S(X)$. 
\item  For every $x,r \in X$ and every $Y,Z \in S(X)$ such that $Z \leq Y$, we have
\begin{equation}\label{eq:H_decreasing}
H(x,r,Y)\leq H(x,r,Z). 
\end{equation} 
\item For every $x,r \in X$ and every  $R>0$, we have 
\begin{align}\label{eq:H_lambda_BQ_N}
\lim_{N \to \infty} \sup \Big \{   |H(x,r,Z+\lambda BQ_N)-H(x,r,Z)|: \  |Z_{00}|\leq R, \ |\lambda| \leq R \Big \}=0.
\end{align}
\item For every $R>0$ there exists a modulus of continuity $\omega_R$ such that
\begin{align}\label{eq:H_norm_-1}
 H \left (z,\frac{B(z-y)}{\varepsilon},Z \right )-H \left (y,\frac{B(z-y)}{\varepsilon},Y \right )  \geq -\omega_R\left (|z-y|_{-1}\left(1+\frac{|z-y|_{-1}}{\varepsilon}\right) \right) 
 \end{align}
for every $\varepsilon>0$, $y,z \in X$ such that $|y|_X,|z|_X\leq R$, $Y,Z \in \mathcal{S}(X)$ satisfying
$$Y=P_N Y P_N \quad Z=P_N Z P_N$$
and 
\begin{align*}
\frac{3}{\varepsilon}\left(\begin{array}{cc}B P_{N} & 0 \\ 0 & B P_{N}\end{array}\right) \leq\left(\begin{array}{cc}Y & 0 \\ 0 & -Z\end{array}\right) \leq \frac{3}{\varepsilon}\left(\begin{array}{cc}B P_{N} & -B P_{N} \\ -B P_{N} & B P_{N}\end{array}\right).
\end{align*}
\item
 If $C>0$ is the constant in  \eqref{eq:b_sublinear} and \eqref{eq:G_bounded}, then, for every $x \in X, p, q \in X,  Y,Z \in \mathcal{S}(X)$,
\begin{align}\label{eq:Hamiltonian_local_lip}
| H(x, r+q,Y+Z)-H(x, r, Y)| \leq C\left(1+|x|_X\right)|q|_X+\frac{1}{2}C^2\left(1+|x|_X\right)^{2}|Z_{00}|.
\end{align}
\end{enumerate}
\end{lemma}

The HJB equation associated with the optimal control problem is the infinite dimensional PDE
\begin{equation}
\label{eq:HJB}
%\left\{\begin{array}{l}
\rho v(x) - \cp{\tilde{\mathcal A} x}{Dv(x)}_X + H(x,Dv(x),D^2v(x))=0,
\quad \forall x \in X.
%0 \leq t \leq T\\[8pt]
%v(T) = \varphi,
%end{array}\right.
\end{equation}
We recall the definition of $B$-continuous viscosity solution from \cite{fgs_book}.
\begin{definition}\label{def:test_functions}
\begin{itemize}
\item[(i)]  $\phi \colon X \to \mathbb R$ is a regular test function if
\begin{small}
\begin{align*}
\phi \in \Phi := \{ \phi \in C^2(X): \phi \textit{ is }B\textit{-lower semicontinuous and }\phi, D\phi , D^2\phi , A^*  D\phi \textit{ are uniformly continuous on }X\};
\end{align*}
\end{small}
\item[(ii)] $g \colon X \to \mathbb R$  is a radial test function if
\begin{align*}
g \in \mathcal G:= \{ g \in C^2(X): g(x)=g_0(|x|_X) \textit{ for some } g_0 \in C^2([0,\infty)) \textit{  non-decreasing}, g_{0}'(0)=0 \}.
\end{align*}
\end{itemize}
\end{definition}
\begin{flushleft}
Note that, if $g\in\mathcal{G}$, we have 
\end{flushleft}
\begin{align}\label{eq:gradient_radial}
D g(x)=\left\{\begin{array}{l}
g_0^{\prime}(|x|_{X}) \frac{x}{|x|_{X}}, \quad \ \ \ \mbox{if} \  x \neq 0, \\
0,  \quad\quad  \ \ \ \ \ \ \ \ \ \ \ \ \,\,\,  \ \mbox{if} \ x=0.
\end{array}\right.
\end{align}
\begin{definition}\label{def:viscosity_solution}
\begin{enumerate}[(i)]
\item A locally bounded $B$-upper semicontinuous function $v:X\to\R$ is a viscosity subsolution of \eqref{eq:HJB} if, whenever $v-\phi-g$ has a local maximum at $x \in X$ for $\phi \in \Phi, g \in \mathcal G$, then 
\begin{equation*}
\rho v(x) - \cp{ x}{\tilde{\mathcal A}^* D\phi(x)}_{X} + H(x,D\phi(x)+Dg(x) ,D^2\phi(x)+D^2g(x))\leq 0.
\end{equation*}
\item 
A locally bounded $B$-lower semicontinuous function $v:X\to\R$ is a viscosity supersolution of \eqref{eq:HJB} if, whenever $v+\phi+g$ has a local minimum at $x \in X$ for $\phi \in \Phi$, $g \in \mathcal G$, then 
\begin{equation*}
\rho v(x) + \cp{ x}{\tilde{\mathcal A}^* D\phi(x)}_{X} + H(x,-D\phi(x)-Dg(x) ,-D^2\phi(x)-D^2g(x))\geq 0.
\end{equation*}
\item 
A viscosity solution of \eqref{eq:HJB} is a function $v:X\to\R$ which is both a viscosity subsolution and a viscosity supersolution of \eqref{eq:HJB}.
\end{enumerate}
\end{definition}
Define $\mathcal{S}:=\{u \colon X \to \mathbb{R}:  \exists k\geq 0 \ \mbox{satisfying \eqref{eq:k_set_uniqueness} and } \tilde C\geq 0 \,\mbox{such that}\,  |u(x)|\leq  \tilde C (1+|x|_X^k)\},$
where
\begin{small}
\begin{equation}\label{eq:k_set_uniqueness}
\begin{cases}
k<\frac{\rho}{C+\frac{1}{2} C^{2}}, \ \ \ \ \ \ \ \ \ \ \ \ \ \ \ \quad \mbox{ if } \ \frac{\rho}{C+\frac{1}{2} C^{2}} \leq 2, \\
C k+\frac{1}{2} C^{2} k(k-1)<\rho,  \quad \mbox{ if } \frac{\rho}{C+\frac{1}{2} C^{2}}>2, 
\end{cases}
\end{equation}
\end{small}
and $C$ is the constant appearing in \eqref{eq:b_sublinear} and \eqref{eq:G_bounded}.

We can now state the theorem characterizing $V$ as the unique viscosity solution of \eqref{eq:HJB} in $\mathcal{S}$. 
\begin{theorem}\label{th:existence_uniqueness_viscosity_infinite}
Let Assumptions \ref{hp:state}, \ref{hp:cost}, and \ref{hp:discount} hold.
The value function $V$ is the unique viscosity solution of \eqref{eq:HJB} in the set $\mathcal{S}$.
\end{theorem}
\begin{proof}
Notice that $V \in \mathcal S$ by Proposition \ref{prop:growth_V}. 

The proof of the fact that $V$ is the unique viscosity solution of the HJB equation can be found in \cite[Theorem 3.75]{fgs_book} as all assumptions of this theorem are satisfied due to Lemma \ref{lemma:properties_H}. 
\end{proof}
\begin{remark}
We remark that, similarly to \cite{defeo_federico_swiech}, Theorem \ref{th:existence_uniqueness_viscosity_infinite} also holds in the deterministic case, i.e. when $\sigma(x,u)=0$.  (in which case we may take $\rho_0=Cm$ and $k<\rho/C$ in \eqref{eq:k_set_uniqueness}). The theory of viscosity solutions handles well degenerate HJB equations, i.e. when the Hamiltonian satisfies
$$H(x,p,Y) \leq H(x,p,Z) $$
for every $Y,Z \in S(X)$ such that $Z\leq Y$. Hence viscosity solutions can be used in connection with the dynamic programming method for optimal control of stochastic differential equations in the case of degenerate noise in the state equation, in particular, when it completely vanishes (deterministic case). This is not possible using the mild solutions approach (see \cite{defeo_federico_swiech}, \cite{fgs_book}).
 \end{remark}
 \begin{remark}
 In this work we could not prove the partial differentiability of $V$ with respect to $x_0$ as in \cite[Theorem 6.5]{defeo_federico_swiech}. Indeed in \cite[Theorem 6.5]{defeo_federico_swiech} a key assumption was
\begin{equation}\label{eq:local_lipschitz_V_norm_-1}
|V(y)-V(x)| \leq K_R|y-x|_{-1}, \quad \forall x, y \in X,|x|_X,|y|_X \leq R.
\end{equation}
In \cite{defeo_federico_swiech} this condition holds under some standard assumptions, see \cite[Example 6.2]{defeo_federico_swiech}. In particular the cost  $l(\cdot,u)$ is assumed to be Lipschitz (uniformly in $u$).\\
However in the present paper we could not prove \eqref{eq:local_lipschitz_V_norm_-1}: indeed due  to Remark \ref{rem:x_0_leq_x_-1} the following inequality (which holds true in \cite{defeo_federico_swiech})  is false 
\begin{equation}\label{eq:x_0_leq_x_-1_n2}
|x_0| \leq C |x|_{-1} \quad \forall x=(x_0,x_1) \in X.
\end{equation} 
This means that even for a Lipschitz $l(\cdot,u)$ (uniformly in $u$) we cannot use \eqref{eq:x_0_leq_x_-1_n2} in the following way
\begin{equation}\label{eq:L_lip_norm_-1}
|L(x, u)-L(y, u)|=|l(x_0,u)-l(y_0,u)| \leq |x_0-y_0| \leq C|x-y|_{-1} \quad \forall x,y  \in X,
\end{equation}
to get the Lipschitzianity of $L$ with respect to $|\cdot|_{-1}$ (as it was done in \cite{defeo_federico_swiech}). 
The best we could get in this paper is \eqref{eq:L_unif_cont_norm_B}, i.e.
$$|L(x, u)-L(y, u)| \leq \omega_R\left(|x-y|_{-1}\right),$$
but this is of course not enough in order to prove \eqref{eq:local_lipschitz_V_norm_-1}.
Hence we could not proceed as in the proof of \cite[Theorem 6.5]{defeo_federico_swiech}.

We finally remark that the same reason prevented us to apply  $C^{1,1}-$regularity results from  \cite{defeo_swiech_wessels}, where $L(\cdot,u)$ was assumed to be  Lipschitz with respect to $|\cdot|_{-1}$, uniformly in  $u.$
  \end{remark}
 
 \section{Applications}\label{sec:advertising}
 In this section we provide applications of our results  to problems coming from economics.
\subsection{Optimal advertising with delays}
The following model is a generalization of the ones in \cite{gozzi_marinelli_2004}, \cite{gozzi_marinelli_savin_2006} to the case of controlled diffusion. We recall that in \cite{defeo_federico_swiech} the case with no delays in the control (i.e. $p_1=0$) was treated. 
 
 The model for the dynamics of the stock of advertising goodwill $y(t)$ of the product is given by the following controlled SDDE
\begin{equation*}
\begin{cases}
dy(t) = \ds \left[a_0 y(t)+b_0 u(t )+ \int_{-d}^0 a_1(\xi)y(t+\xi)\,d\xi +
            \int_{-d}^0p_1(\xi)u(t+\xi)\,d\xi\right]dt +\left[  \sigma_0 + \gamma_0 u(t) \right] dW(t), \quad t \geq 0, \\[10pt]
y(0)=\eta_0, \quad y(\xi)=\eta_1(\xi),\;
u(\xi)=\delta(\xi),\;\; \xi\in[-d,0),
\end{cases}
\end{equation*}
where $d>0$, the control process $u(t)$ models the intensity of advertising spending and $W$ is a real-valued Brownian motion. Moreover
\begin{enumerate}[(i)]
\item $a_0 \leq 0$ is a constant factor of image deterioration in absence of advertising;
\item   $b_0 \geq 0$ is a constant advertising effectiveness factor;
\item $a_1 \leq 0$ is a given deterministic function satisfying the assumptions used in the previous sections and represent the distribution of the forgetting time;
\item $ p_1 \geq 0$ is a given deterministic function satisfying the assumptions used in the previous sections and it is the density function of the time lag between the advertising expenditure and the corresponding effect on the goodwill level; 
\item $\sigma_0>0$ is   a fixed uncertainty level in the market;
\item $\gamma_0> 0$ is a constant uncertainty factor which multiplies the advertising spending; 
\item $\eta_0 \in \mathbb R$ is the level of goodwill at the beginning of the advertising campaign;
\item $\eta_1 \in L^2([-d,0];\mathbb R)$ is the history of the goodwill level.
\item $\delta \in L^2([-d,0]; U)$ is the history of the advertising spending.
\end{enumerate}
Again, we use the same setup of the stochastic optimal control problem as the one in Section \ref{sec:formul} and the control set $U$ is here
$U= [0,\bar u]$
for some $\bar u>0$.
The optimization problem is 
$$
\inf_{u \in \mathcal U}\mathbb{E} \left[\int_0^\infty e^{-\rho s} l(y(s),u(s)) d s\right],
$$
where  $\rho >0$ is a discount factor, $l(x,u)=h(u)-g(x)$, with  a continuous and convex cost function $h \colon U \rightarrow \mathbb R$ and a continuous and concave utility function $g \colon \mathbb R \rightarrow \mathbb R$ which satisfies Assumption \ref{hp:cost}. 

We are then  in the setting of Section \ref{sec:B_continuity}. Therefore we can use Theorem \ref{th:existence_uniqueness_viscosity_infinite} to characterize the value function $V$ as the unique viscosity solution to \eqref{eq:HJB}.

\subsection{Optimal investment models with time-to-build}
The following model is inspired by  \cite[p. 36]{fabbri_federico}. See also, e.g., \cite{bambi_fabbri_gozzi (2012), bambi_digirolami_federico_gozzi (2017)} for similar models in the deterministic setting.

Let us consider a state process $y(t),$ representing the  stock capital of a certain enterprise at time $t$, and a control process $u(t) \geq 0$, representing the investment undertaken at time $t$ to increase $y$. We assume that the dynamics of $y(t)$  is given by the following SDDE
\begin{small}
\begin{equation*}
\begin{cases}
d y(t)= \left [b_0 (u(t))+ \int_{-d}^0 p_1(\xi) u (t+\xi) d \xi  \right ] d t+\sigma_0( u(t) )d W(t), \quad t \geq 0, \\[10pt]
y(0)=\eta_0, \quad
u(\xi)=\delta(\xi),\;\; \xi\in[-d,0),
\end{cases}
\end{equation*}
\end{small}
where 
\begin{enumerate}[(i)]
\item $b_0 \colon U \to [0,\infty)$ is a continuous bounded function representing an instantaneous effect of the investment on the capital;
\item $ p_1 \geq 0$ is a given deterministic function satisfying the assumptions used in the previous sections and representing the density function of the time-to-build between the investment and the corresponding effect on the stock capital;
\item $\sigma_0 \colon U \to [0,\infty)$ is a a continuous bounded function representing the uncertainty of achievement of the investment plans. 
\item $\eta_0 \in \mathbb R$ is the initial level of the capital;
\item $\delta \in L^2([-d,0];U)$ is the history of the investment spending.
\end{enumerate}
Again, we use the same setup of the stochastic optimal control problem of Section \ref{sec:formul} and the control set $U$ here is
$U= [0,\bar u]$
for some $\bar u>0$.
The goal is to maximize, over all $u (\cdot) \in \mathcal U,$ the expected integral of the discounted future profit flow in the form
$$
\mathbb{E}\left[\int_0^{\infty} e^{-\rho t}( F(y(t))-C(u(t))) d t\right],
$$
where $F: \mathbb{R} \rightarrow \mathbb{R}$ is a production function and $C: U \rightarrow \mathbb{R}$ is a cost function. We assume that $F,C$ satisfy Assumption \ref{hp:cost}.  The optimization problem is equivalent to minimize, over all $u (\cdot) \in \mathcal U,$
$$
\mathbb{E} \left[\int_0^\infty e^{-\rho t} l(y(t),u(t)) d t\right],
$$
where   $l(x,u):=C(u)- F(x) $. We are then  in the setting of Section \ref{sec:formul} (with $a_1=0$). Therefore we can use Theorem \ref{th:existence_uniqueness_viscosity_infinite} to characterize the value function $V$ as the unique viscosity solution to \eqref{eq:HJB}.

\begin{flushleft}
\textbf{Acknowledgments.} The author is very grateful to Andrzej Święch for many useful conversations related to the content of the present paper. He also thanks the two anonymous Referees for their careful reading of the manuscript and their useful comments and suggestions.
\end{flushleft}
\begin{flushleft}
\textbf{Conflict of interest statement.} The author has no competing interests to declare that are relevant to the content of this article.

\textbf{Funding.} This research was partially financed by the INdAM (Instituto Nazionale di Alta Matematica F. Severi) - GNAMPA (Gruppo Nazionale per l'Analisi Matematica, la Probabilità e le loro Applicazioni) Project “Riduzione del modello in sistemi dinamici stocastici in dimensione infinita a due scale temporali” and by
the Italian Ministry of University and Research (MIUR), within PRIN project 2017FKHBA8 001 (The
Time-Space Evolution of Economic Activities: Mathematical Models and Empirical Applications). 
\end{flushleft}
 
\bibliographystyle{amsplain}

\end{document}